\documentclass[12pt]{article}
\usepackage{preprint-layout}

\newcommand{\addressumushort}{Mathematics and Mathematical Statistics, Ume{\aa}~University, Sweden}

\newcommand{\addressjushort}{Mechanical Engineering, J\"onk\"oping University, Sweden}

\newcommand{\addressuclshort}{Mathematics, University College London, UK}

\newcommand{\tn}{|\kern-1pt|\kern-1pt|}

\newcommand{\mcT}{\mathcal{T}}

\newcommand{\mcO}{\mathcal{O}}

\DeclareMathOperator{\Div}{Div}
\newcommand{\IR}{\mathbb{R}}

\title{Cut Finite Elements for Convection in Fractured Domains}
\date{}
\author{Erik Burman,
Peter Hansbo,
Mats G. Larson and
Karl Larsson
}

\begin{document}
\maketitle
\begin{abstract}
We develop a cut finite element method (CutFEM) for the convection problem in a so called fractured domain which is a union of manifolds of different dimensions such that a $d$ dimensional component always resides on the boundary of a $d+1$ dimensional component. This type of domain can for instance be used to model porous media with embedded fractures that may intersect. The convection problem is formulated in a compact form suitable for analysis using natural abstract directional derivative and divergence operators. The cut finite element method is posed on a fixed background mesh that covers the domain and the manifolds are allowed to cut through a fixed background mesh in an arbitrary way. We consider a simple method based on continuous piecewise linear elements together with weak enforcement of the coupling conditions and stabilization. We prove a priori error estimates and present illustrating numerical examples.
\end{abstract}

\section{Introduction}
\paragraph{Fractured Domains.}
Transport phenomena in media with complicated microstructure occur in several 
applications for instance transport in porous media  and composite materials. The 
properties of the microstructure may have different characteristics ranging from 
stochastic to highly structured or a combination of these. In this work we focus on 
problems where the microstructure consists of embedded surfaces and their 
intersections. The surfaces can be used to model fractures or thin embedded 
sheets with different transport properties. We refer to such domains as fractured 
domains, see examples in Figure~\ref{fig:fractured-schematic}.

\begin{figure}
\centering
\begin{subfigure}[t]{.3\linewidth}\centering
\includegraphics[width=0.8\linewidth]{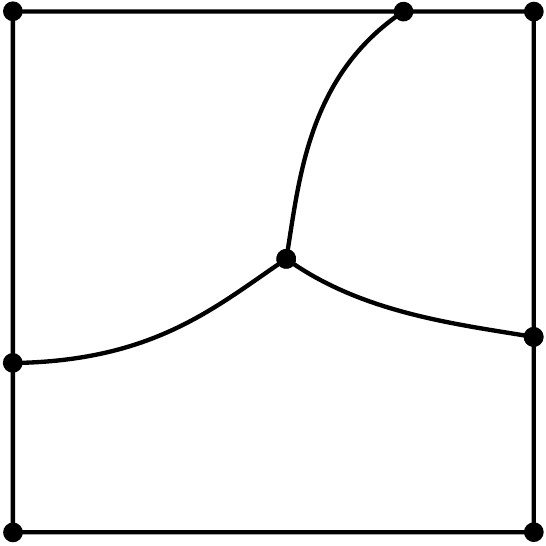}
\end{subfigure}
\begin{subfigure}[t]{.3\linewidth}\centering
\includegraphics[width=0.8\linewidth]{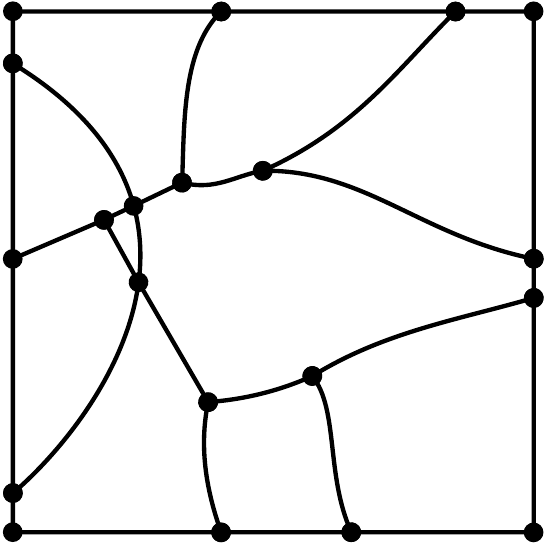}
\end{subfigure}
\caption{Two example fractured domains in 2D.}
\label{fig:fractured-schematic}
\end{figure}

\paragraph{New Contributions.}
A fractured domain in $\IR^n$ is a disjoint union of smooth manifolds of dimension 
$d=0,\dots,n$, constructed in such a way that a $d$ dimensional component 
always reside on the boundary of a $d+1$ dimensional component. These domains 
are also called mixed-dimensional or stratified domains. See the recent work 
\cite{BoNoVa17} where a similar description is used to study the pressure problem. 
On such a domain we consider a first order system of hyperbolic equations which 
models transport in fractured media. 

Introducing  convenient multi-dimensional directional derivative and divergence 
operators the problem may be formulated in an abstract form similar to the standard 
one field transport problem on a domain in $\IR^n$. 

We develop a cut finite element method, see \cite{BurClaHanLarMas15} for an introduction, 
which is based on embedding the composite domain into a fixed background mesh and then 
for each of the components we define the active mesh as the set of all elements that 
intersect the component. Note that in this way we obtain one active mesh for each 
component of the domain and thus certain elements will appear in several meshes. 
The active meshes are each equipped with a continuous finite element space and 
the finite element method is obtained by stabilizing the variational formulation using 
certain stabilization terms. Other methods, for instance the
discontinuous Galerkin method, may 
be used as well but here we stay in the simplest framework of
continuous finite element spaces and Galerkin least-squares stabilization.

Using the abstract framework the formulation of the method is straightforward and 
the basic coercivity result also follows. Combining coercivity with
the consistency of the finite element method and applying interpolation 
error estimates for cut finite element methods on embedded manifolds 
\cite{BurHanLarMas16}, we obtain a priori error estimates that are
optimal in the sense typical for stabilized finite element methods
applied to the transport equation.

\paragraph{Earlier Work.} 
The computation of flows in fractured media has received increasing
attention lately. For modelling of the equations of flow and transport
in porous media we refer to \cite{MR1911534,MR2359657} and 
in particular the mixed dimensional models presented in
\cite{NoBo17}, and \cite{BoNoVa17}.

Finite volume approaches have been proposed
\cite{2016arXiv160106977B,MR3489127,MR3671645} and virtual elements 
in  \cite{FumKei17}.  For stochastic 
methods we refer to \cite{MR3372046,MR3595292}. Various model reduction
techniques have been proposed such as
\cite{MR3489127,MR3307584}. Other work considers meshed fractures
\cite{MR2400465} or particle methods \cite{MR3437907}.

Compared to meshed methods cut finite element methods have the advantage 
that we do not need to construct a mesh that fits a possibly complex arrangement 
of fractures. For surface 
surface partial differential equations this approach, also called trace finite 
elements, was first introduced in \cite{OlReGr09} and has then been developed 
in different directions including stabilization \cite{BurHanLar15} higher order approximations in \cite{Reu15}, 
discontinuous Galerkin methods \cite{BurHanLarMas17}, transport problems 
\cite{OlsReuXu14-b,2015arXiv151102340B}, embedded membranes \cite{CenHanLar16}, 
coupled bulk-surface problems \cite{BurHanLarZah16} and \cite{GroOlsReu15}, 
minimal surface problems \cite{CenHanLar15}, and time 
dependent problems on evolving surfaces \cite{HanLarZah16,OlsReu14,OlsReuXu14}, and \cite{Zah17}. We also refer to the overview article 
\cite{BurClaHanLarMas15} and the references therein. For the present
work we also draw on experiences from the paper \cite{BurHanLarMas16},
where CutFEMs on on embedded manifolds of arbitrary codimensions was
considered and \cite{MR3709202} for the design of CutFEMs on
composite surfaces. 

\paragraph{Outline.} In Section 2 we introduce the notion of fractured domains, the 
abstract differential operators on these domains, and formulate an integration by parts 
formula, and formulate the model problem both in componentwise and abstract form. 
In Section 3 we formulate the finite element method. In Section 4 we derive a priori 
error estimates. In Section 5 we present numerical results. In Section 6 we draw some 
conclusions and mention directions for future work.

\section{The Model Problem}

\subsection{The Domain and Function Spaces}
\label{sec:defs}

We here introduce the notation needed to describe a fractured domain and define
the appropriate function spaces on such a domain.

\paragraph{Composite Domain.} Let $\Omega$ be a 
domain in $\IR^n$ such that  
\begin{itemize}
\item There is a partition $\mcO = \{\Omega_d\}_{d=0}^n$,
\begin{equation}
\Omega = \cup_{d=0}^n \Omega_d
\end{equation}
\item For each $\Omega_d \in \mcO$ there is a partition $\mcO_d  = \{ \Omega_{d,i} \}_{i=1}^{n_d}$,
\begin{equation}
\Omega_d = \cup_{i=1}^{n_d} \Omega_{d,i} 
\end{equation}
where each $\Omega_{d,i}$ is a smooth $d$-dimensional manifold with 
boundary $\partial \Omega_{d,i}$.
\item The partition satisfies 
\begin{equation}
 \partial \Omega_{d,i} \subset \cup_{l=1}^{d-1} \Omega_{l}  \qquad i = 1,\dots, n_d, \quad d =0,\dots,n
\end{equation}
\item We define the boundary operators 
\begin{equation}
\partial_d \mcO_{d} = \bigsqcup_{i=1}^{n_d} \partial \Omega_{d,i}
\qquad  
\partial \mcO = \bigsqcup_{d=0}^n \partial \mcO_{d}
\end{equation}
where $\sqcup$ denotes the disjoint union.
\end{itemize}
The notation introduced here for partitions in $\mcO$ is illustrated in Figure~\ref{fig:notation}.

\begin{figure}
\centering
\includegraphics[width=0.4\linewidth]{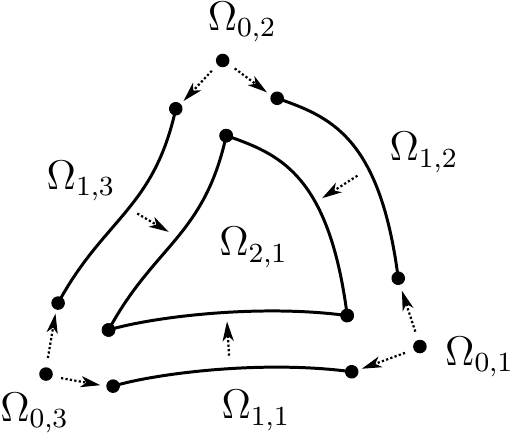}
\caption{Illustration of notation used for components of different dimension in a fractured domain in $\IR^n$, $n=2$, where a $d<n$ dimensional component always resides on the boundary of a $d+1$ dimensional component.}
\label{fig:notation}
\end{figure}

\paragraph{Function Spaces on $\mcO$ and $\partial \mcO$.} 
\begin{itemize}
\item Let $H^s(\Omega_{d,i})$ be the Sobolev space on the manifold $\Omega_{d,i} 
\in \mcO$ of order $s$ with scalar product $(v,w)_{H^s(\Omega_{d,i})}$, 
and define  
\begin{equation}
H^s(\mcO_{d}) = \bigoplus_{i=1}^{n_d} H^s(\Omega_{d,i}) ,
\qquad 
H^s(\mcO) = \bigoplus_{d=0}^n H^s(\Omega_d)
\end{equation} 
with scalar products 
\begin{equation}
(v_d,w_d)_{H^s(\mcO_d)} = \sum_{i=1}^{n_d} (v,w)_{H^s(\Omega_{d,i})}, 
\qquad
(v,w)_{H^s(\mcO)} = \sum_{d=0}^n (v_d,w_d)_{H^s(\Omega_d)}
\end{equation}
and inner product norms $\| v_d \|_{H^s(\mcO_d)}$ and $\|v\|_{H^s(\mcO)}$.  For 
$d=0$, $H^s(\Omega_{0,i}) = \IR$ and is equipped with the usual absolute value  
$\| v \|^2_{H^s(\Omega_{0,i})} = v^2$.

\item In the case $s=0$ we use the notation $L^2(\mcO_d)=H^0(\mcO_d)$ 
and $L^2(\mcO)=H^0(\mcO)$ with scalar products $(v_d,w_d)_{\mcO_d}$ 
and $(v,w)_\mcO$ and norms $\|v\|_{\mcO_d}$ and $\| v \|_\mcO$

\item On $\partial \mcO$ we define 
$L^2(\partial \mcO) = \bigoplus_{d=1}^n \bigoplus_{i=1}^{n_d} L^2(\partial \Omega_{d,i})$ and 
we equip the components in $\partial \mcO_d$ with the natural $d-1$ dimensional 
measure and thus all components of dimension less or equal to $d-2$ has measure 
zero which means that 
\begin{equation}
(v,w)_{\partial \mcO_d} = \sum_{i=1}^{n_d} (v,w)_{\partial \Omega_{d,i}} = 
\sum_{i=1}^{n_d} (v,w)_{\partial \Omega_{d,i}\cap \Omega_{d-1}}
\end{equation}

\end{itemize}
\paragraph{Tangential and Normal Vector Fields.}
\begin{itemize}
\item We say that $a= \oplus_{d=0}^n a_{d}$ is a tangential vector field on $\mcO$ 
if $a_d = \oplus_{i=1}^{n_d} a_{d,i}$  and each $a_{d,i}$ is a tangential vector field 
on the manifold $\Omega_{d,i} \in \mcO_d$.

\item We define the unit exterior normal vector field $\nu$ on $\partial \mcO$ 
by $\nu|_{\partial \Omega_{d,i}} = \nu_{d,i}$, where $\nu_{d,i}$ is the unit 
tangential vector field on $\Omega_{d,i}$ which is orthogonal to $\partial \Omega_{d,i}$ 
and exterior to $\Omega_{d,i}$, see Figure~\ref{fig:normals}.

\item The pointwise dot product $a \cdot b$ of two tangential vector fields $a$ 
and $b$ on $\mcO$ is the scalar field $(a\cdot b)_d = a_d \cdot b_d$ on each 
$\mcO_d$, $d=0,\dots,n$.
\end{itemize}
\paragraph{Tangential Gradient.}
\begin{itemize}
\item For $\delta>0$ let $U^n_\delta(\Omega_{d,i}) = \cup_{x\in \Omega_{d,i}} B_\delta(x) \subset \IR^n$, where $B_\delta(x)$ is the open ball of radius $\delta$ withe center $x$, be an open neighborhood of $\Omega_{d,i}$. Then there is a continuous extension operator 
$E:v \in H^s(\Omega_{d,i}) \rightarrow H^s(U^n_\delta(\Omega_{d,i}))$, see 
\cite{BuHaLaLaMa15} for the construction necessary to handle the fact that 
$\Omega_{d,i}$ has a boundary. We employ the 
shorthand notation $Ev = v^e$ when necessary for clarity otherwise we simplify 
further and write $v = v^e$.
 
\item Let $\nabla_d$ be the tangential gradient on $\Omega_d$ and 
\begin{equation}
\nabla v = \oplus_{d=1}^{n} \nabla_d v_d
\end{equation}
where for each $x \in \Omega_{d,i}$, $(\nabla_d v)|_x = (P_d \nabla_{\IR^n} v^e)|_x $ 
and $P_d|_x: \IR^n \rightarrow T_x(\Omega_{d,i})$ is the  projection onto the tangent plane $T_x(\Omega_{d,i})$. 

\item Given a tangential vector field $\beta$ let 
\begin{equation}
V_{\beta} = \{ v \in L^2(\mcO) : \|\beta \cdot \nabla v\|_{\mcO} \lesssim 1 \} 
\end{equation}
In other words we for $v\in V_{\beta}$ in  each component $\Omega_{d,i}\in\mcO$ have that $v|_{\Omega_{d,i}} \in L^2(\Omega_{d,i})$ and $v|_{\Omega_{d,i}} \in H^1(\omega)$ where $\omega$ is any $d-1$ dimensional manifold $\omega$ tangential to $\beta|_{\Omega_{d,i}}$.
\end{itemize}

\begin{figure}
\centering
\begin{subfigure}[t]{.4\linewidth}\centering
\includegraphics[width=0.87\linewidth]{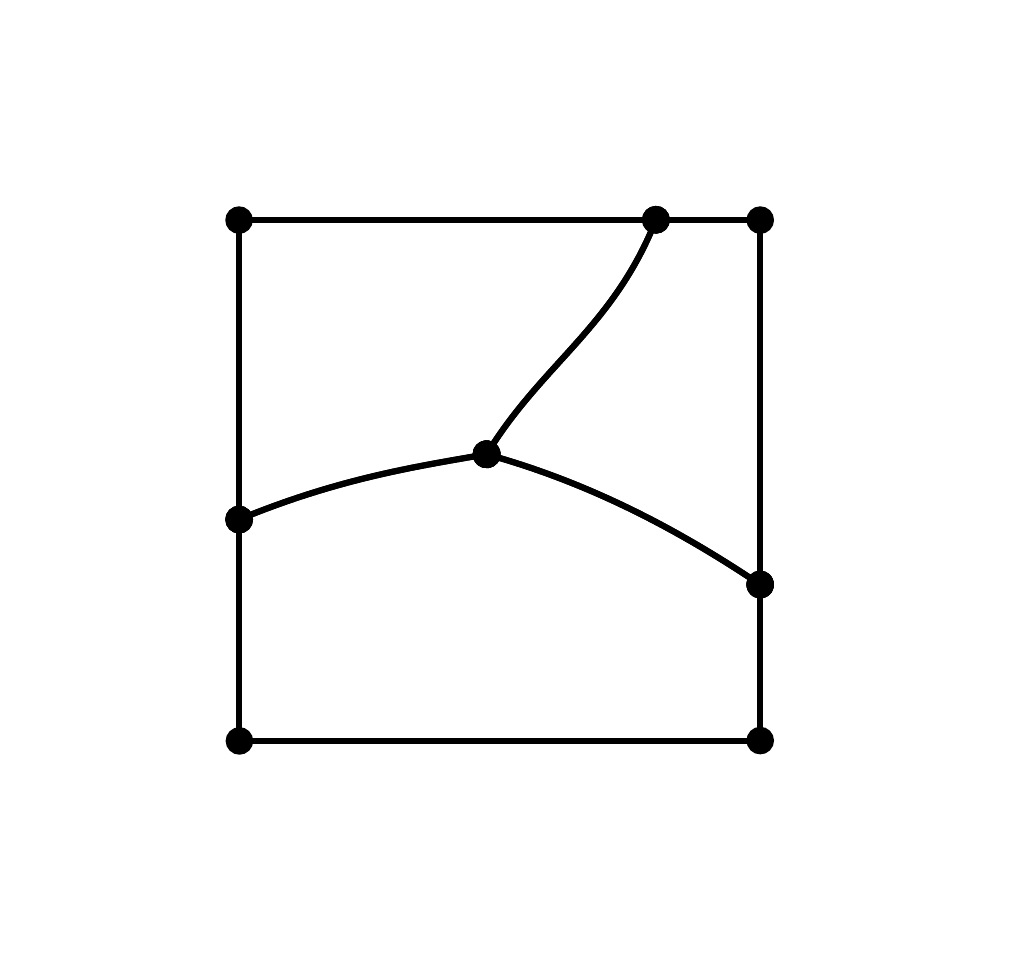}
\caption{Fractured domain} \label{fig:domain-ex-a}
\end{subfigure}
\begin{subfigure}[t]{.4\linewidth}\centering
\includegraphics[width=0.87\linewidth]{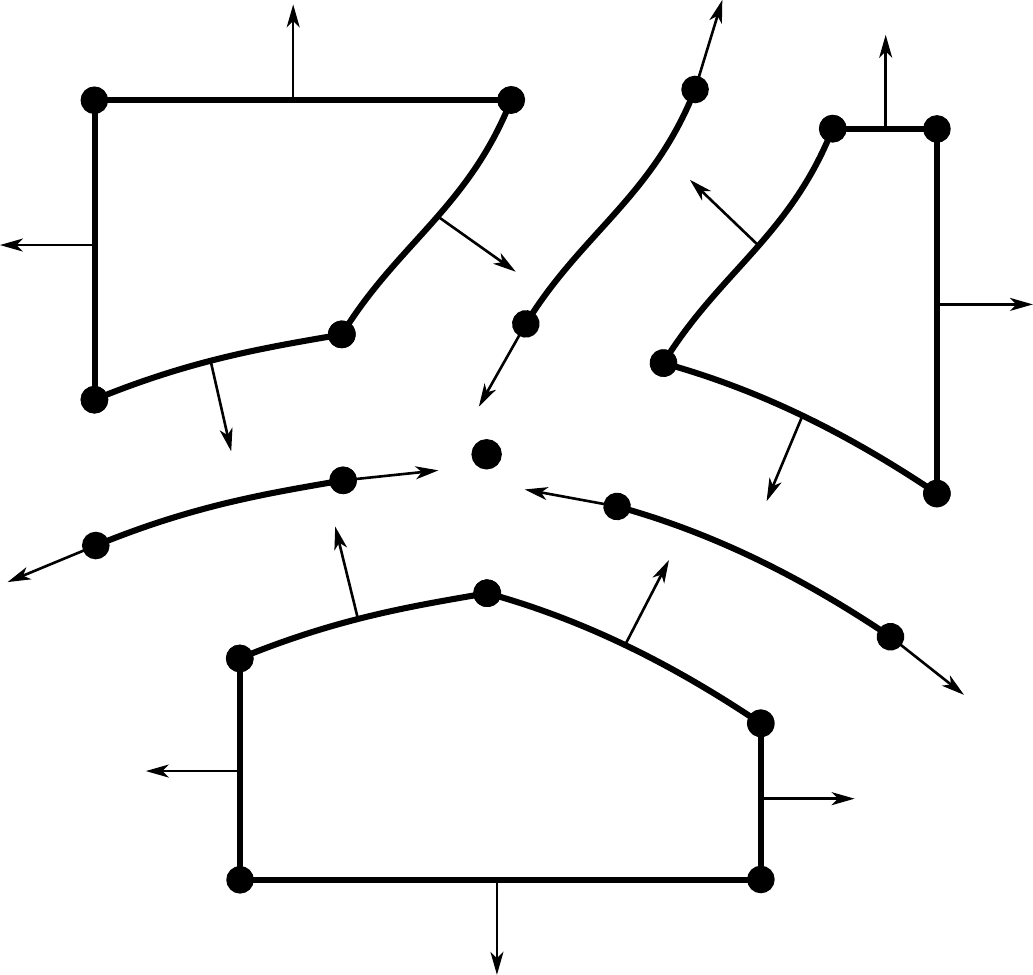}
\caption{Exterior unit normal field} \label{fig:domain-ex-b}
\end{subfigure}
\caption{Illustration of the exterior unit normal vector field on a fractured domain. (a) A fractured domain with $n=2$, $n_0=8$, $n_1=10$, and $n_2=3$. (b) The exterior unit normal field for the domain in (a).}
\label{fig:normals}
\end{figure}

\subsection{Abstract Differential Operators}

In this section we introduce jump operators used for coupling between different subdomains and
also differential operators that enable formulation of the convection problem in a compact form.

\paragraph{Jump Operators.} To express the coupling between subdomains we use the following operators:
\begin{itemize}
\item The jump operator $\llbracket \cdot\rrbracket_d:L^2(\partial \mcO_{d+1} ) 
\rightarrow L^2(\mcO_d)$ is defined by $\llbracket \cdot\rrbracket_n 
=  0$ and for $d=0,\dots,n-1$, 
\begin{equation}
\llbracket  v_{d+1} \rrbracket_{d}|_{\Omega_{d,i}} 
= \sum_{j=1}^{n_{d+1}} v_{d+1,j} |_{\partial \Omega_{d+1} \cap \Omega_{d,i}} 
\end{equation} 
We then have the identity 
\begin{equation}
(v_{d+1},w_d)_{\partial \mcO_{d+1}} = (\llbracket v_{d+1} \rrbracket_d, w_d)_{\mcO_d}
\qquad w_d \in L^2(\mcO_d)
\end{equation}
and we also note that 
\begin{equation}
\llbracket v_{d+1} w_d \rrbracket_d =  \llbracket v_{d+1}  \rrbracket_d w_d \qquad w_d \in L^2(\mcO_d)
\end{equation}

\item The jump operator $[\cdot ]_d: L^2(\mcO_{d-1}) \times L^2(\partial \mcO_{d}) 
\rightarrow L^2(\partial \mcO_{d})$ is defined by $[v]_0 = 0$ and for $d=1,\dots,n$, 
\begin{equation}
[v]_{d,i} |_{\partial \Omega_{d,i}} = v_{d,i}|_{\partial \Omega_{d,i}} - \sum_{j=1}^{n_{d-1}}v_{d-1,j}|_{\partial\Omega_{d,i}\cap\Omega_{d-1,j}} 
\end{equation}

\end{itemize}
Note that the jump operators provide the coupling between the different 
subdomains and that only neighboring subdomains with difference in dimension 
equal to one couple to each other.

%

\paragraph{The Directional Derivative and Divergence Operators.} 
Let $\beta$ be a smooth tangential vector field on $\mcO$, i.e. 
$(\beta)_{d,i}$ is a smooth tangential vector field on each $\Omega_{d,i} \in \mcO$,
 and let $\nu$ be the unit exterior normal vector field on $\partial \mcO$ defined 
 in Section \ref{sec:defs}.
\begin{itemize}
\item Let the derivative $D_\beta$ in the direction $\beta$ be defined by
\begin{equation}
( D_{\beta} v )_n = \beta_n \cdot \nabla_n v_n, 
\qquad 
(D_{\beta} v )_0 = \sum_{i=1}^{n_1} \nu_{1,i} \cdot \beta_{1,i} (v_{0} - v_{1,i})  
\end{equation}
and for $d=1,\dots,n-1,$ let
\begin{equation}
(D_{\beta} v)_d
= \beta_d \cdot \nabla_d v_d + \sum_{i=1}^{n_{d+1}} 
\nu_{d+1,i} \cdot \beta_{d+1,i} (v_{d} - v_{d+1,i})  
\end{equation}
or equivalently in terms of the jump operators
\begin{equation}
(D_\beta v)_d = \beta_d \cdot \nabla_d v_d - \llbracket \nu_{d+1} \cdot \beta_{d+1} [v]_{d+1}\rrbracket_{d}
\end{equation}
\item Let  the divergence $\Div \beta$ be defined by
\begin{equation}\label{eq:Div}
(\Div\beta)_d = \nabla_d \cdot \beta_d - \sum_{i=1}^{n_{d+1}} 
\nu_{d+1,i} \cdot \beta_{d+1,i}
\end{equation}
or equivalently in terms of the jump operators
\begin{equation}\label{eq:Div-jump}
(\Div\beta)_d = \nabla_d \cdot \beta_d 
- \llbracket \nu_{d+1} \cdot \beta_{d+1} \rrbracket_d
\end{equation}

\end{itemize}
%
%
%
In order to formulate a partial integration formula for $D_\beta$ we introduce the 
notation
\begin{equation}
\partial \mcO_B = \partial \mcO \cap \partial \Omega 
= 
\sqcup_{i,d} (\partial \Omega_{i,d} \cap \partial \Omega), 
\qquad  
\partial \mcO_I = \partial \mcO \setminus \partial \Omega 
= \sqcup_{i,d} (\partial \Omega_{i,d} \setminus \partial \Omega)
\end{equation}
to denote the components in $\partial \mcO$ which belong to the boundary and the interior respectively. We end this section by stating a lemma from \cite{BuHaLaLa18-b}.

\begin{lem}\label{lem:partial-integration} {\bf(Partial Integration)}
For  a smooth tangential vector field $\beta$ on $\mcO$ and $v\in V_\beta$,
\begin{equation}\label{eq:Div-formula}
\Div(\beta v ) = D_\beta v + (\Div\beta) v 
\end{equation}
and for $v,w\in V_\beta$,  
\begin{equation}\label{eq:Dbeta-partial-integration}
(D_\beta v, w)_\mcO =  - (v,D_\beta w)_{\mcO}  -((\Div\beta) v,w)_{\mcO} 
+ (\nu \cdot \beta [v],[w])_{\partial \mcO_I}
+ (\nu \cdot \beta v, w)_{\partial \mcO_B}
\end{equation}
where $\nu$ is the exterior unit normal vector field on $\partial \mcO$.
\end{lem}
\subsection{The Model Problem}

In this section we introduce our model convection problem on a fractured domain.

\paragraph{Componentwise Formulation.}
Find $u_{d,i}:\Omega_{d,i} \rightarrow \IR$ such that
\begin{alignat}{3}\label{eq:strong-form-bulk}
\nabla_{d,i}\cdot (\beta_{d,i} u_{d,i} )  + \alpha_{d,i} u_{d,i}
- \llbracket \nu_{d+1} \cdot \beta_{d+1} u_{d+1}\rrbracket_{d,i} 
&= f_{d,i}& \qquad &\text{in $\Omega_{d,i}$}
\\ \label{eq:strong-form-inflow-internal}
(\nu_{d,i} \cdot \beta_{d,i})_- [ u ]_{d,i} &= 0 & \qquad &\text{on $\partial \Omega_{d,i} \setminus \partial \Omega$}
\\ \label{eq:strong-form-inflow-boundary}
(\nu_{d,i} \cdot \beta_{d,i})_-  (u_{d,i} - g_{d,i}) &=0 & \qquad  &\text{on 
$ \partial \Omega_{d,i} \cap \partial \Omega$}
\end{alignat}
where 
$(v)_- = \min(v,0)$ denotes the negative part of 
$v$.

\paragraph{Abstract Formulation.} 
We note that using the definition (\ref{eq:Div}) of the divergence we may rewrite 
(\ref{eq:strong-form-bulk}) as follows
\begin{align}\nonumber
&\nabla_{d,i}\cdot (\beta_{d,i} u_{d,i} )  + \alpha_{d,i} u_{d,i}
- \llbracket \nu_{d+1} \cdot \beta_{d+1} u_{d+1}\rrbracket_{d,i} 
\\
&\qquad =
\beta_{d,i} \cdot \nabla_{d,i} u_{d,i} 
+ (\nabla_{d,i}\cdot \beta_{d,i} )  u_{d,i} + \alpha_{d,i} u_{d,i}
\\ \nonumber
&\qquad \qquad - \llbracket \nu_{d+1} \cdot \beta_{d+1} (u_{d+1} - u_d) \rrbracket_{d,i} 
- \llbracket \nu_{d+1} \cdot \beta_{d+1} u_d \rrbracket_{d,i} 
\\
&\qquad =
\beta_{d,i} \cdot \nabla_{d,i} u_{d,i} - \llbracket \nu_{d+1} \cdot \beta_{d+1} [u]_{d+1} \rrbracket_{d,i} 
\\ \nonumber
&\qquad \qquad + (\nabla_{d,i}\cdot \beta_{d,i} )  u_{d,i} - \llbracket \nu_{d+1} \cdot \beta_{d+1}  \rrbracket_{d,i} u_d
+ \alpha_{d,i} u_{d,i}
\\
&\qquad = (D_\beta u + \Div \beta + \alpha )_{d,i}
\end{align}
where we essentially added and subtracted $\llbracket \nu_{d+1} \cdot \beta_{d+1} u_d \rrbracket$ 
and rearranged the terms. Thus in terms of the abstract operators (\ref{eq:strong-form-bulk}) 
takes the form
\begin{equation}
D_\beta u  + (\alpha + \Div\beta )u = f
\end{equation}
Thus we obtain the problem: find $u \in V$ such that 
\begin{alignat}{3}\label{eq:problem-a}
D_\beta u + \gamma u &=f &   \qquad &\text{in $\mcO$}
\\ \label{eq:problem-b}
(\nu \cdot \beta)_- [ u ] &= 0 & \qquad &\text{on $\partial \mcO_I$}
\\ \label{eq:problem-c}
(\nu \cdot \beta)_- ( u - g) &= 0 & \qquad &\text{on $\partial \mcO_B$}
\end{alignat}
where $\gamma = \alpha + \Div \beta$ or in component form
\begin{equation}
\gamma_d = \alpha_d + \nabla_d \cdot \beta_d -  \llbracket \nu_{d+1} \cdot \beta_{d+1}   \rrbracket_{d}
\end{equation}

\paragraph{Weak Formulation.} Find $u \in V_\beta$ such that 
\begin{equation}\label{eq:weak-problem}
a(u,v)  = l(v) \qquad \forall v \in V_\beta
\end{equation}
where the forms are defined by
\begin{align}
a(v,w) &= (D_\beta v,w)_\mcO + (\gamma v,w)_\mcO 
+ (|\nu \cdot \beta|_- [ v ],[w])_{\partial \mcO_I}
+ (|\nu \cdot \beta|_- v , w)_{\partial \mcO_B}
\\
l(w) &= (f,w)_\mcO +  (|\nu \cdot \beta|_- g, w)_{\partial \mcO_B}
\end{align}
and we used the simplified notation $| v |_- = |(v)_-|$, for the absolute value of the 
negative part. Using Lemma \ref{lem:partial-integration} we may derive the following stability 
result. 

\begin{lem}\label{lem:coercivity}{\bf (Coercivity)}
If there is a constant $c_0>0$ such that 
\begin{equation}\label{eq:assumption-coeff}
c_0 \leq \| 2 \alpha + \Div \beta \|_{L^\infty(\Gamma)} 
\end{equation}
then 
\begin{equation}\label{eq:coercivity}
\| v \|^2_{\mcO} 
+ \|[v]\|^2_{|\nu \cdot \beta|,\partial \mcO_I} 
+ \| v \|^2_{|\nu \cdot \beta|,\partial \mcO_B} 
\lesssim
a(v,v)
\qquad \forall v\in V_\beta
\end{equation}
where we introduced the norms
\begin{equation}
\| w \|^2_{|\nu \cdot \beta|,\partial \mcO_J} =  \left\| ( |\nu \cdot \beta|_{-} )^{1/2} w \right\|^2_{\partial \mcO_J} ,\quad J\in\{I,B\}
\end{equation}
\end{lem}

\section{The Cut Finite Element Method}
\subsection{The Mesh and Finite Element Spaces}
\begin{itemize}
\item Let $\Omega_0 \in \IR^n$ be a polygonal domain such that $\Omega\subset \Omega_0$ 
and let $\{ \mcT_{h,0},\, h\in (0,h_0] \}$ for some constant $h_0 > 0$ be a family of quasi-uniform 
meshes with mesh parameter $h$ of $\Omega_0$ om shape regular elements $T$.

\item Let $V_{h,0}$ be a finite element space of continuous piecewise polynomial functions on $\mcT_{h,0}$. We consider low order elements with linear polynomials or tensor product polynomials. Adaption to higher order elements is outlined in the next section.

\item For each $\Omega_{d,i} \in \mcO$ let the active mesh be defined by
\begin{equation}\label{eq:mesh-active}
\mcT_{h,d,i} = \{ T \in \mcT_{h,0} : T\cap \Omega_{d,i} \neq \emptyset \} 
\end{equation}
and define the associated finite element space $V_{h,d,i} = V_{h,0} |_{\Omega_{d,i}}$, see Figure~\ref{fig:meshes}. Note that in most cases it is not necessary to introduce active meshes on components without a source term that constitute part of the boundary. This is due to the solution in those parts being directly given by either the boundary condition or the coupling to a higher dimensional component. For simplicity, we therefore from this point on assume all components $\Omega_{d,i} \in \mcO$ satisfy $\Omega_{d,i} \cap \partial\mcO_B = \emptyset$.

\item Define the finite element space on $\mcO$ as the direct sum
\begin{equation}
V_h = \bigoplus_{d=0}^n V_{h,d}, \qquad V_{h,d} = \bigoplus_{i=1}^{n_d} V_{h,d,i}
\end{equation} 
\end{itemize}

\begin{figure}
\centering
\begin{subfigure}[t]{.31\linewidth}\centering
\includegraphics[width=0.9\linewidth]{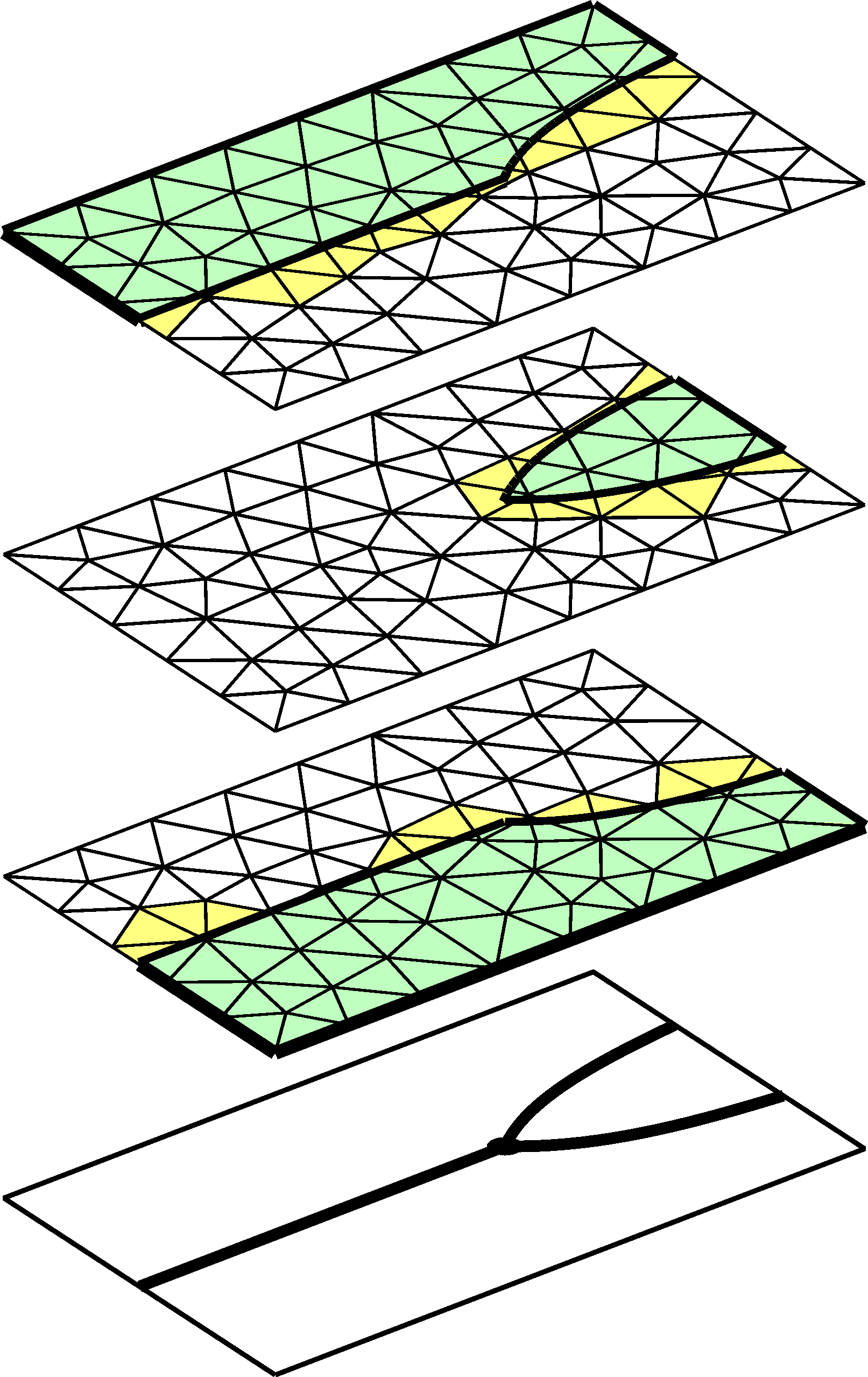}
\caption{$d=2$} \label{fig:meshes-codim0}
\end{subfigure}
\begin{subfigure}[t]{.31\linewidth}\centering
\includegraphics[width=0.9\linewidth]{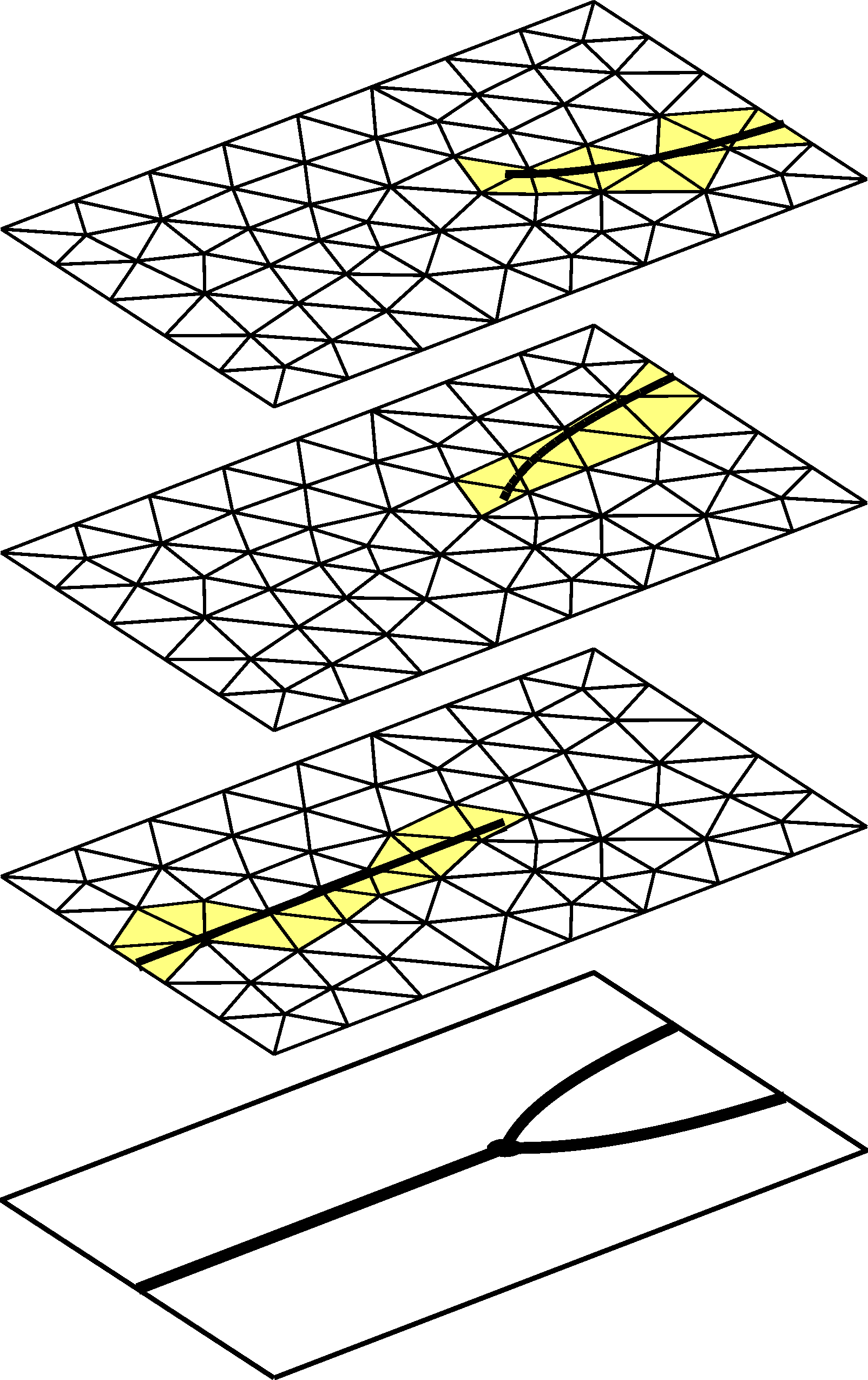}
\caption{$d=1$}\label{fig:meshes-codim1}
\end{subfigure}
\begin{subfigure}[t]{.31\linewidth}\centering
\includegraphics[width=0.9\linewidth]{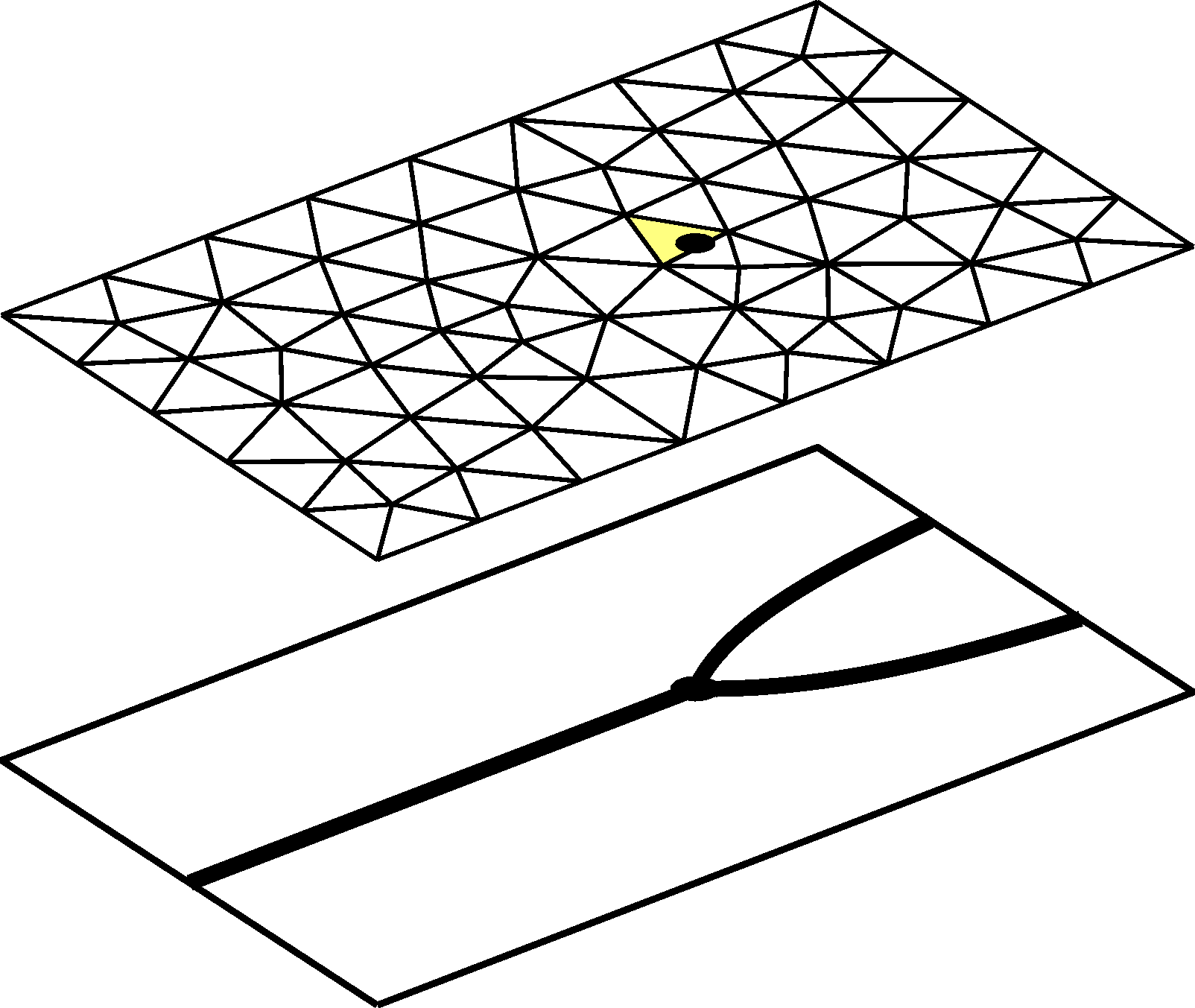}
\caption{$d=0$}\label{fig:meshes-codim2}
\end{subfigure}
\caption{Meshes for an example geometry in 2D consisting of three bulk domains ($d=2$), three cracks ($d=1$), and one bifurcation point ($d=0$). The colored parts are the active meshes $\{\mcT_{h,d,i}\}$.}
\label{fig:meshes}
\end{figure}

\subsection{The Method} 
We consider a finite element method based on the weak formulation (\ref{eq:weak-problem}) which takes care of the coupling between the different domains. Using a conforming 
finite element space we will need to stabilize 
the convection term and furthermore since we are using a cut finite element method 
we need to stabilize in order to control the variation of the solution orthogonal to 
$\Omega_{d,i}$. For simplicity, we will consider piecewise linear elements and use  
standard Galerkin Least Squares (GLS) method together with so called full gradient 
stabilization for the cut elements developed in \cite{BuHaLaMaZa16}.  The full gradient 
stabilization adds control of the variation of the finite element solution in the direction 
orthogonal to the manifold $\Omega_{d,i}$ and also provides control of the resulting 
condition number of the linear system of equations. The full gradient stabilization is not 
consistent and we scale it in such a way that we do not lose order of convergence. 
Essentially, for linear elements we obtain an artificial tangent diffusion of order $h^{3/2}$. 
In the case of higher order elements we may use a weaker full gradient stabilization or 
preferably  a more refined stabilization which is consistent (on exact geometry) such 
as the recently developed normal stabilization, \cite{BurHanLarMas16} and \cite{GraLehReu16}, or 
the combined normal-face stabilization \cite{LaZa17}.

\paragraph{Galerkin Least Squares (GLS).}
Find $u_h \in V_h$ such that 
\begin{equation}\label{eq:fem}
a_h(u_h,v) = l_h(v)\qquad \forall v \in V_h
\end{equation}
where 
\begin{align}
a_h(v,w) &= \sum_{d=0}^n \sum_{i=1}^{n_d} a_{h,d,i}(v_{d,i},w_{d,i})
+ (|\nu\cdot \beta|_- [v],[w])_{\partial \mcO_I} +  (|\nu\cdot \beta|_- v,w)_{\partial \mcO_B} 
\\
l_h(v) &= \sum_{d=0}^n  \sum_{i=1}^{n_d} l_{h,d,i}(v_{d,i})
\end{align}
The forms $a_{h,d,i}$ and $l_{d,h,i}$ are linear forms on $V_{h,d,i}$ defined by
\begin{align}
a_{h,d,i}(v,w) 
&= 
(L_d v, w)_{\Omega_{d,i}} + (\tau_1 h L_d v, L_d w )_{\Omega_{d,i}} 
+ s_{h,d,i}(v,w)
\\
l_{h,d,i}(v) &= (f_{d,i}, v_{d,i})_{\Omega_{d,i}} 
+ (\tau_1 h f_{d,i}, L_d v)_{\Omega_{d,i}} 
+ (|\nu\cdot \beta|_{-} g,v)_{\partial \Omega_{d,i}} 
\end{align}
where $\tau_1>0$ is a parameter 
\begin{align}
L_{d,i} v &= (D_\beta v + \gamma v)|_{d,i}
\\
&=\beta_{d,i}\cdot \nabla_{d,i} v_{d,i}  
+ ( (\nabla_{d,i} \cdot (\beta_{d,i} ) + \alpha_{d,i}) v_{d,i}
- \llbracket \nu_{d+1} \cdot \beta_{d+1} v_{d+1}\rrbracket_{d,i} 
\end{align}
and  $s_{h,d,i}$ is the stabilization form
\begin{equation}
s_{h,d,i}(v,w) = \tau_2 h^{3 - (n-d)} (\nabla_{\IR^n} v, \nabla_{\IR^n} w )_{\mcT_{h,d,i}}
\end{equation}
where $\tau_2$ is a parameter and $\nabla_{\IR^n}$ denotes the gradient in $\IR^n$. 
We also note that $n-d$ is the codimension of $\Omega_{d,i}$ and thus the scaling 
factor $h^{-(n-d)}$ compensates for the fact that we integrate over the $n$ dimensional 
set $\mcT_{h,d,i}$. We will see that the additional $h^3$ scaling ensures that we do not 
lose  order of convergence when adding $s_{h}$.

%
%
%


\section{Error Estimates} 

We prove a basic error estimate in the natural energy norm associated with 
the GLS method. We assume that the geometry is represented exactly and 
that all integrals are computed exactly. In this situation the proof is done using 
the standard techniques combined with an interpolation error estimate for 
manifolds of arbitrary codimension. Estimates of the geometric error can be 
done using a generalization of the approach developed in \cite{BurHanLarMas16}. 

\subsection{Coercivity and Continuity}
Let 
\begin{equation}
V^e = \{v^e = Ev \,:\, v \in V_\beta \} \,, \qquad W = V^e + V_h
\end{equation}
where $E$ is the extension operator defined in Section~\ref{sec:defs} when introducing the tangential gradient.
Define the energy norm 
\begin{equation}\label{eq:energy-norm}
\tn v \tn^2_h = \| v \|^2_\mcO + h \| L v \|^2_\mcO + \| v \|^2_{s_h} 
+ \|[v]\|^2_{|\nu \cdot \beta|,\partial \mcO_I} 
+ \| v \|^2_{|\nu \cdot \beta|,\partial \mcO_B}, 
\qquad v \in W 
\end{equation}
and the norm
\begin{equation}\label{eq:energy-norm-cont}
\tn v \tn^2_{h,*} = h^{-1} \| v \|^2_\mcO+\tn v \tn^2_h, 
\qquad v \in W
\end{equation}
which we will need in the statement of continuity.

\begin{lem} 
The form $a_h$ is continuous
\begin{equation}\label{eq:Ah-continuity}
a_h(v, w) \lesssim \tn v \tn_{h,*} \tn w \tn_h, \qquad v,w \in W
\end{equation}
and if (\ref{eq:assumption-coeff}) holds coercive 
\begin{equation}\label{eq:Ah-coercivity}
\tn v \tn_h^2 \lesssim a_h(v,v)\qquad v,w \in W
\end{equation}
\end{lem}
\begin{proof} The continuity (\ref{eq:Ah-continuity}) follows by 
first applying the Cauchy-Schwarz inequality in all the symmetric 
terms of $a_h$,
\begin{equation}
a_h(v, w) \lesssim  (D_\beta v,w)_{\mcO} + \tn v \tn_h \tn w \tn_h
\end{equation}
Using the integration by parts formula in the first term of the right
hand side yields
\begin{align}
 (D_\beta v,w)_{\mcO} &= -((\Div\beta) v,w)_{\mcO}  - (v,D_\beta w)_{\mcO} 
+ (\nu \cdot \beta [v],[w])_{\partial \mcO_I}
+ (\nu \cdot \beta v, w)_{\partial \mcO_B}
\\
&\leq  (v,D_\beta w)_{\mcO} + C \tn v \tn_h \tn w \tn_h
\\
&\leq h^{-1/2} \|v\|_\mcO h^{1/2} \|L w\|_\mcO + C \tn v \tn_h \tn w \tn_h
\\
&\lesssim  \tn v \tn_{h,*} \tn w \tn_h
\end{align} 
where we used the uniform bound $\|\Div\beta \|_{L^\infty(\mcO)} \lesssim 1$ 
and the definition of the norm $ \tn \cdot \tn_{h,*}$ in the last step.

The coercivity (\ref{eq:Ah-coercivity}) follows by observing that 
\begin{equation}
a_h(v,w) = a(v,w) + (\tau_1 h L v, L w)_{\mcO} + s_h(v,w)
\end{equation} 
and thus 
\begin{align}
a_h(v,v) &= a(v,v) + \tau_1 h \| L v \|^2_\mcO + \| v \|^2_{s_h}
\\
&\gtrsim \| v \|^2_\mcO + \|[v]\|^2_{|\nu \cdot \beta|,\partial \mcO_I} 
+ \| v \|^2_{|\nu \cdot \beta|,\partial \mcO_B} 
+ \tau_1 h \| L v \|^2_\mcO + \| v \|^2_{s_h}
\\
&= \tn v \tn^2_h
\end{align}
where we used Lemma \ref{lem:coercivity}.
\end{proof}

\subsection{Interpolation Error Estimates}
There is an interpolation operator $\pi_h:L^2(\Omega_{d,i}) \rightarrow V_{h,d,i}$ such 
that the following interpolation error estimate holds
\begin{equation}\label{eq:interpol-energy}
\tn v - \pi_h v \tn^2_*\lesssim h^{3} \| v \|^2_{H^{k+1}(\mcO)}
\end{equation}
We define $\pi_h$ by 
\begin{equation}
\pi_h v = \pi_{h,Cl} v^e
\end{equation} 
where $\pi_{h,Cl}: L^2(\mcT_{h,d,i}) \rightarrow V_{h,d,i}$ is the usual Clement interpolator. 
We refer to \cite{BurHanLarMas16} for further details including a proof of the basic interpolation 
estimate 
\begin{equation}
\| u - \pi_h u \|^2_{\Omega_{d,i}} + h^2 \| \nabla_d ( u - \pi_h u )\|^2_{\Omega_{d,i}}
\lesssim h^4 \| u \|^2_{H^2(\Omega_{d,i})}
\end{equation}
which is the used to derive (\ref{eq:interpol-energy}).


\subsection{Error Estimates}
\begin{thm} If $u$ is the solution to (\ref{eq:weak-problem}) satisfies $u \in H^2(\mcO)$ 
and $u_h$ is the finite element approximation defined by (\ref{eq:fem}), then 
\begin{equation}
\tn u - u_h \tn_h^2 \lesssim h^{3}  \| u \|^2_{H^{2}(\mcO)}
\end{equation}
\end{thm}
\begin{proof} Using coercivity 
\begin{align}
\tn u - u_h \tn_h^2 &\lesssim a_h(u - u_h, u - u_h) 
\\
&\lesssim a_h(u - u_h, u - \pi_h u) + a_h(u - u_h, \pi_h u - u_h) 
\\
&\lesssim \tn u - u_h \tn_h \|u - \pi_h u\tn_{h,*} + a_h(u, \pi_h u - u_h) - l_h(\pi_h u - u_h) 
\\
&\lesssim \delta \tn u - u_h \tn^2_h + \delta^{-1} \tn u - \pi_h u\tn^2_{h,*}  + s_h(u^e, \pi_h u - u_h)
\end{align}
for $\delta>0$. Next 
\begin{align}
s_h (u, \pi_h u - u_h) &= s_h (u, \pi_h u - u ) +  s_h (u, u - u_h ) 
\\
&\leq \| u \|_{s_h} \| \pi_h u - u \|_{s_h} + \|u\|_{s_h} \| u - u_h \|_{s_h}
\\
&\leq  \| u \|^2_{s_h} 
+ \underbrace{\| \pi_h u - u \|^2_{s_h}}_{\leq \tn u - \pi_h u \tn^2_{h,*}} 
+\delta^{-1} \|u\|^2_{s_h} 
+ \delta\| u - u_h \|^2_{s_h}
\end{align}
Using kick back and taking $\delta>0$ small enough we arrive at
\begin{equation}
\tn u - u_h \tn_h^2 \lesssim \|u - \pi_h u\tn^2_{h,*}  + \| u \|^2_{s_h} 
\lesssim h^3 \| u \|_{H^2(\mcO)}^2 + h^3 \| u \|^2_{H^1(\mcO)} 
\end{equation}
where we used the interpolation error bound 
(\ref{eq:interpol-energy}) for the first term and the second was estimated as follows  
\begin{equation}
\|  u  \|^2_{s_h} 
= \sum_{d=0}^n \sum_{i=1}^{n_d} \tau_2 h^{3 - (n-d)} \| \nabla_{\IR^n} u^e \|^2_{\mcT_{h,d,i}}
\lesssim \sum_{d=0}^n \sum_{i=1}^{n_d}  h^3 \| \nabla_d u \|^2_{\Omega_{d,i}}
\end{equation}
where we used the estimate 
\begin{equation}
 \| \nabla_{\IR^n} u^e \|^2_{\mcT_{h,d,i}} 
 \lesssim \| \nabla_d u \|^2_{\mcT_{h,d,i}} 
 \lesssim h^{n-d} \| \nabla_d u \|^2_{\Omega_{d,i}}
\end{equation}
which completes the proof.
\end{proof}

\section{Numerical Examples}

\paragraph{Implementation.}
To generate numerical examples we implemented the method \eqref{eq:fem} in 2D, i.e. $n=2$, which means that in our examples the fractured domains may consist of bulk domains ($d=2$), cracks ($d=1$), and bifurcation points ($d=0$).
We first generate a background triangle mesh $\mcT_{h,0}$ embedding the complete geometry and from this mesh we extract an active mesh for each bulk domain, crack domain and bifurcations point, see Figure~\ref{fig:meshes}. On each active mesh we then define a finite element space consisting of linear elements. Note that, while we do generate an active mesh and corresponding linear finite element space for each bifurcation point, this is actually not required as the solution there will only be a point value making it redundant to define a finite element.

\paragraph{Parameters and Meshes.} In all our examples below we use the Galerkin least squares parameter $\tau_1=10^{-2}$, stabilization parameter $\tau_2=10^{-3}$ and $\alpha_{d,i}=0$. Also, in all examples the background mesh $\mcT_{h,0}$ is a triangulation of the unit square $\Omega_0 = [0,1]^2$ with mesh parameter $h=0.1$. The resulting active meshes for Example 1--3 are presented in Figure~\ref{fig:ex1-3-mesh} while the active meshes for Example 4, which also includes a bifurcation point, are presented in Figure~\ref{fig:ex4-mesh}.

\paragraph{Example 1: Crack with in-flow.} This simple example is outlined in Figure~\ref{fig:ex1} where a crack divides the unit square in half. Here the vector fields $\{\beta_{2,i}\}$ in the bulk domains only goes into the crack resulting in the solution on the crack being effected by the bulk solutions but not the other way around. For this example we can actually derive an exact solution where $u=1$ in the bulk domains and $u=2y$ on the crack. As this solution lies in $V_h$ our numerical approximation coincides with the exact solution.

\paragraph{Example 2: Crack with out-flow.} In this example presented in Figure~\ref{fig:ex2} we revert the bulk vector fields in Example 1, yielding a crack with only out-flow to the bulk. As expected the solution in the bulk is affected by the solution on the crack but not the other way around. Also in this case we can derive the exact solution, $u=e^{-2y}$, which is well approximated by our numerical solution.

\paragraph{Example 3: Flow crossing a crack.} In Figure~\ref{fig:ex3} we consider the same geometry as in previous examples but with diagonal bulk vector fields passing through the crack. First, in Figure~\ref{fig:ex3-num1} we consider the case where the vector field in the crack is zero which results in there being no transport in the crack. The presence of the crack in this case actually doesn't effect the solution at all which gives some modeling possibilities as the presence of a crack also allow for discontinuous solutions. Increasing the crack vector field, we note in Figures~\ref{fig:ex3-num2}--\ref{fig:ex3-num3}, that the solution is transported along the crack when passing to the other side.

\paragraph{Example 4: Cracks with a bifurcation point.} This example is presented in Figure~\ref{fig:case4-illustration} and the active meshes used are presented in Figure~\ref{fig:ex4-mesh}. In contrast to previous examples we here also include a bifurcation point where the crack splits. From the top and bottom bulk domains we have flow into the crack while we from the third bulk domain have flow out of the crack. First, in Figure~\ref{fig:case4-same}, we consider the case where the vector fields on the cracks all are unit vectors in the tangential direction. We note that the solution flowing into the bifurcation point is then evenly divided between the two cracks flowing out of the bifurcation point.
In Figure~\ref{fig:case4-diff} we change the relation of the vector fields between the top and bottom cracks, i.e. the cracks flowing out of the bifurcation point, yielding a slightly different distribution. This change also effects the in-flow from the bulk regions which is clear by inspecting the crack solutions further away from the bifurcations point.

\paragraph{Example 5: System of cracks.}
As a final example we in Figure~\ref{fig:ex-tree} consider a system of cracks affected by in-flow from bulk domains. In this case the vector fields on the cracks are again unit vectors in the tangential direction. Thus, at each bifurcation point the sum of the crack solutions flowing into a bifurcation point will equal the sum of the crack solutions flowing out of the bifurcation point.

\section{Conclusions}
We develop a cut finite element method for a convection problem on a fractured domain. 
The upshot of the method is that the mesh does not need to conform to the embedded manifolds, 
which in practice is very convenient. The cut elements are handled using certain stabilization 
terms which leads to a stable method with optimal order convergence properties. Different methods 
may be used to discretize the PDE, and we have here chosen to study a least squares stabilized 
formulation which is convenient to implement and analyze. Some directions for future work 
include existence and uniqueness results for convection problems on fractured domains, 
extensions to convection diffusion problems, higher order methods, time dependent problems, 
and coupled problems with both flow equations and transport.
\begin{figure}
\centering
\begin{subfigure}[t]{.30\linewidth}\centering
\includegraphics[width=0.9\linewidth]{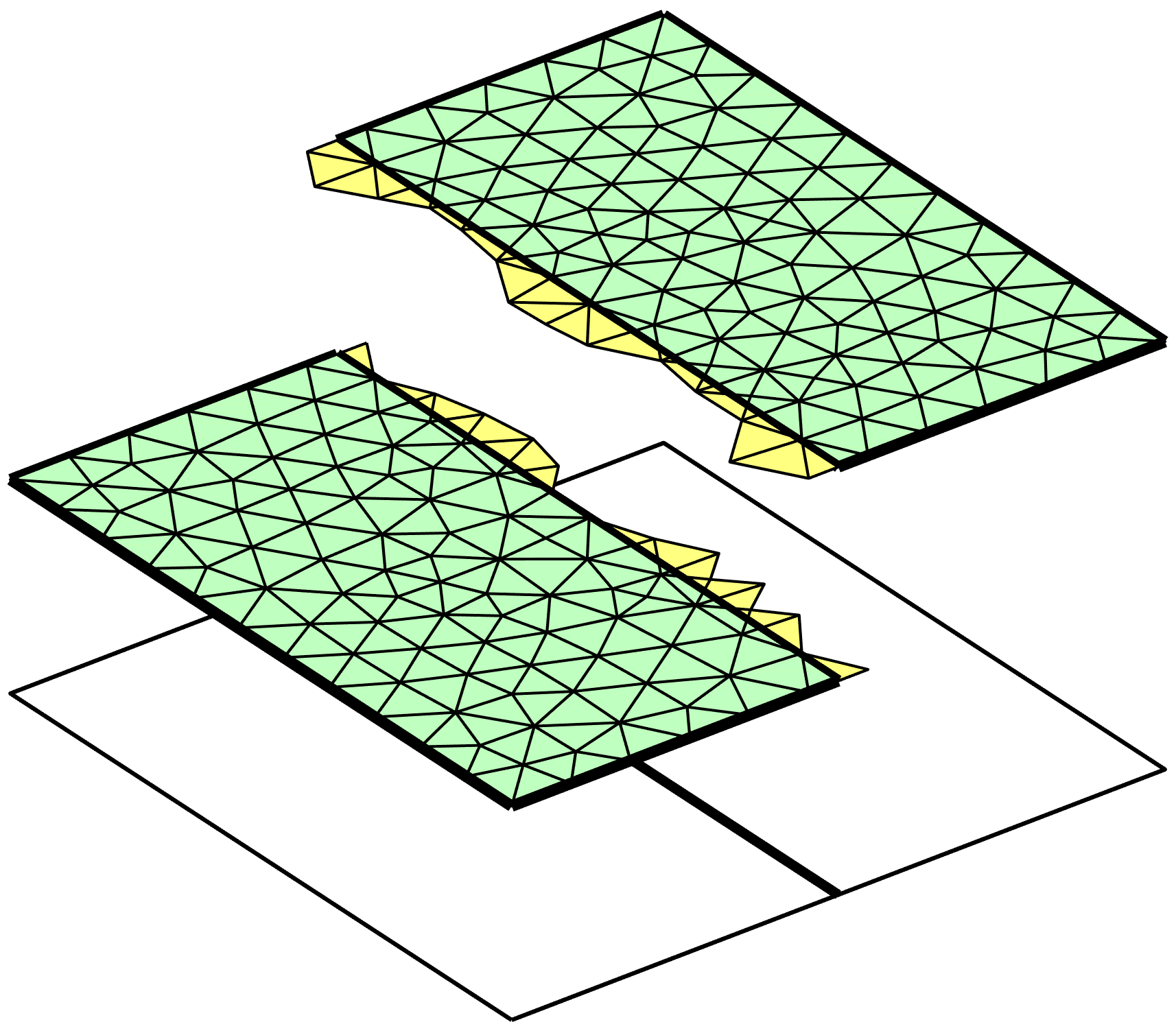}
\caption{$d=2$} \label{fig:case1-mesh-bulk}
\end{subfigure}
\begin{subfigure}[t]{.30\linewidth}\centering
\includegraphics[width=0.9\linewidth]{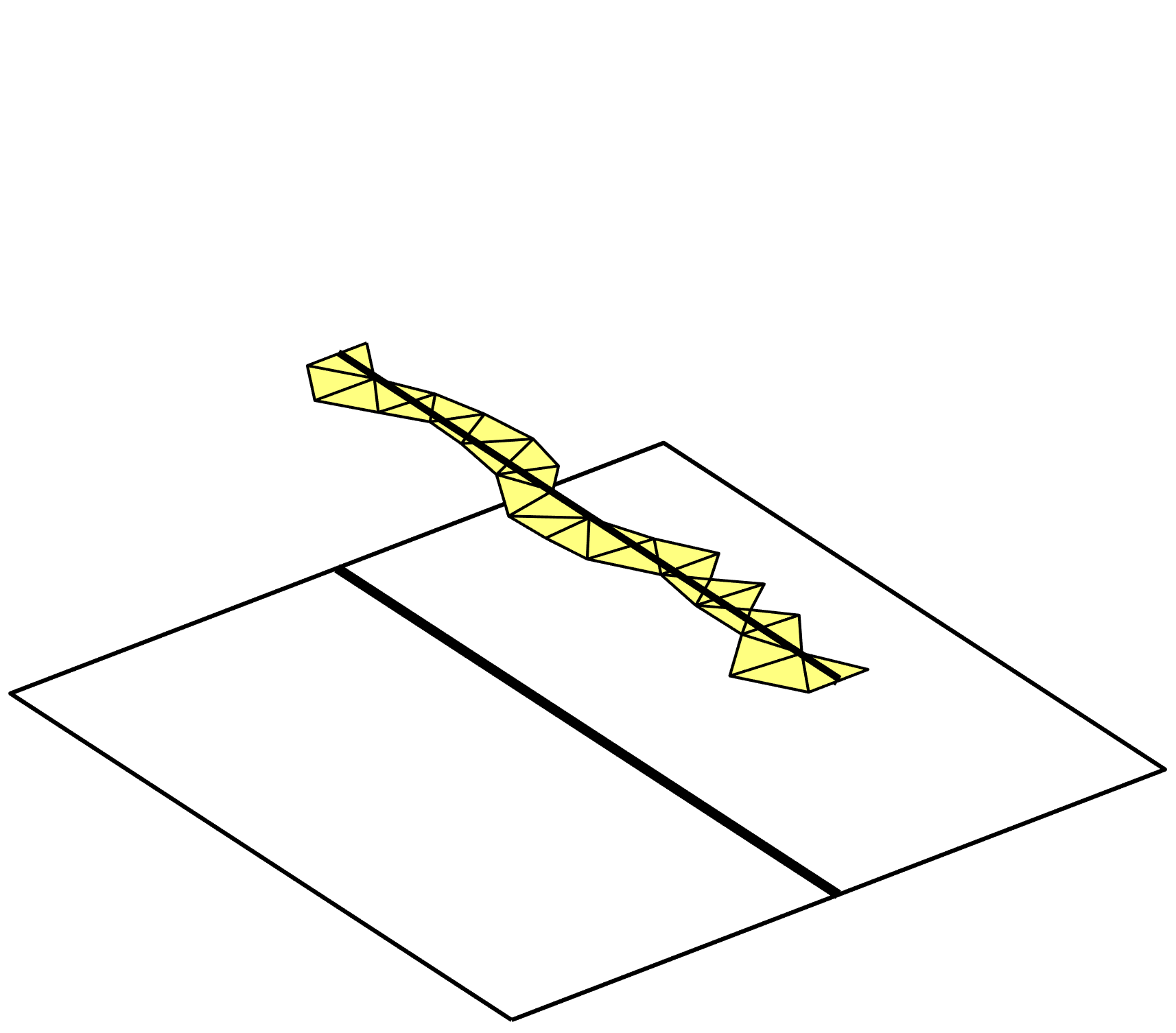}
\caption{$d=1$}\label{fig:case1-mesh-crack}
\end{subfigure}
\caption{Active meshes ($h=0.1$) used for Examples 1--3 where a single crack divides the unit square into two equal parts.}
\label{fig:ex1-3-mesh}
\end{figure}

\begin{figure}
\centering
\begin{subfigure}[t]{.3\linewidth}\centering
\includegraphics[height=0.75\linewidth]{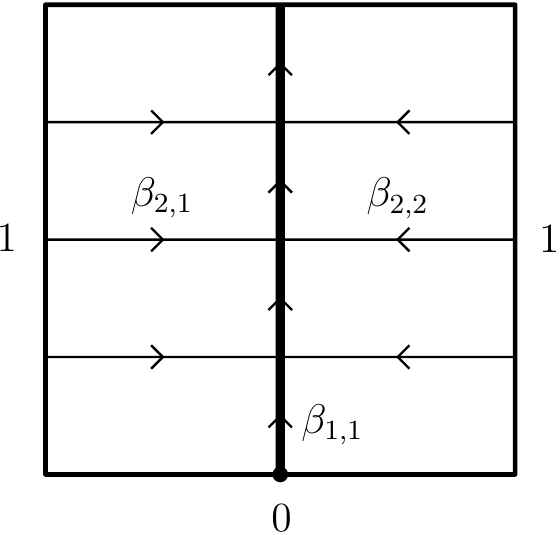}
\caption{Set-up} \label{fig:ex1-illustration}
\end{subfigure}
\begin{subfigure}[t]{.3\linewidth}\centering
\includegraphics[height=0.75\linewidth]{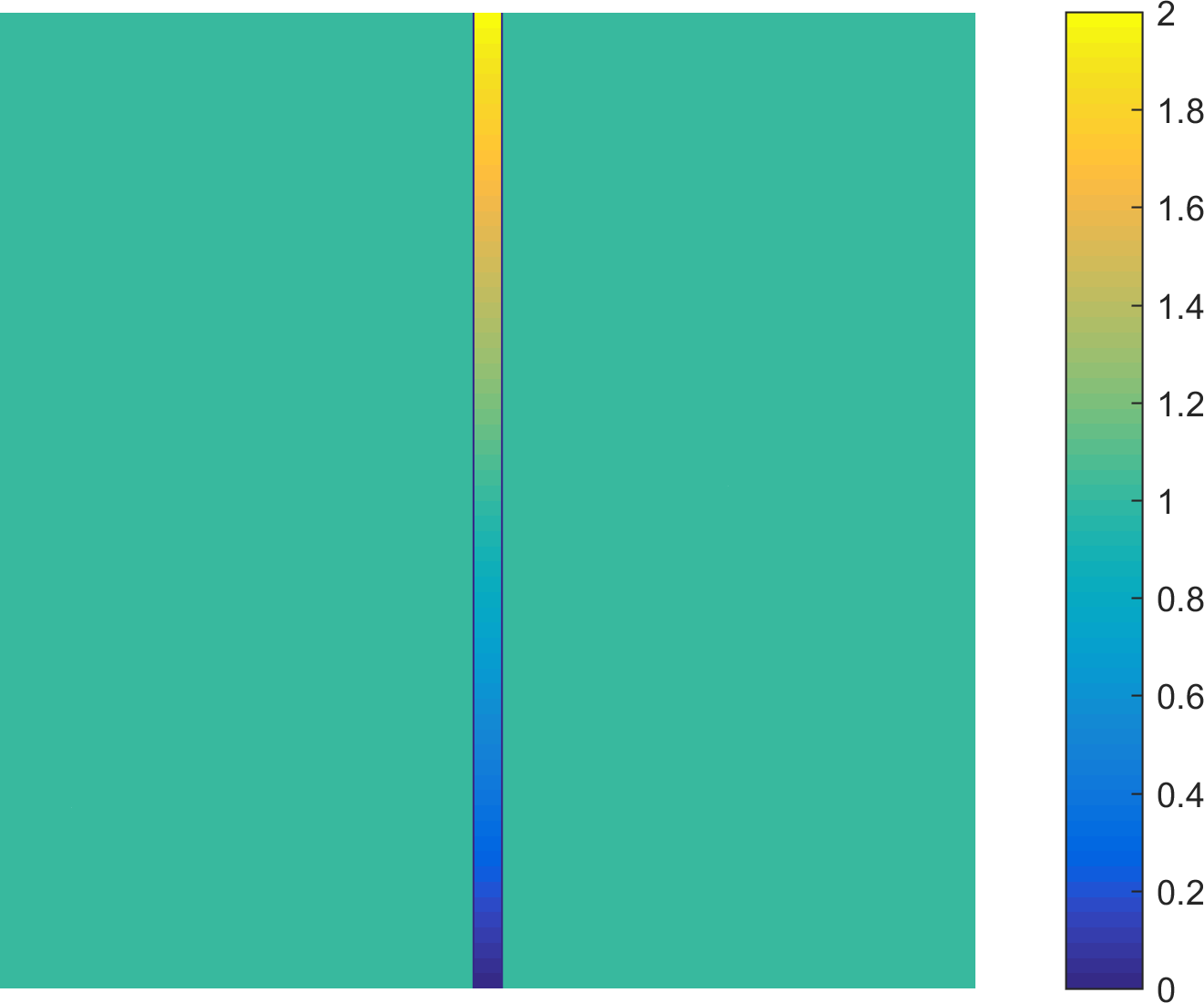}
\caption{Numerical solution}\label{fig:ex1-num}
\end{subfigure}
\caption{Crack with in-flow (Example 1).
(a) The set-up for this example is $\beta_{2,1}=[1,0]$, $\beta_{2,2}=-\beta_{2,1}$ and in the crack $\beta_{1,1} = [0,1]$.
(b) The numerical solution corresponds well to the exact solution which is $u=1$ in the two bulk domains and $u=2y$ on the crack.
}
\label{fig:ex1}
\end{figure}

\begin{figure}
\centering
\begin{subfigure}[t]{.3\linewidth}\centering
\includegraphics[height=0.75\linewidth]{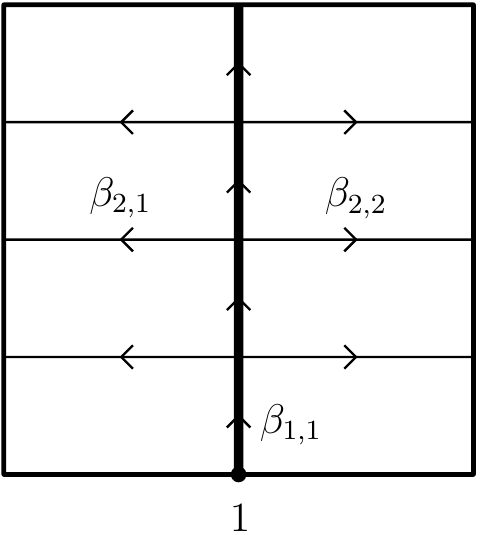}
\caption{Set-up} \label{fig:ex2-illustration}
\end{subfigure}
\begin{subfigure}[t]{.3\linewidth}\centering
\includegraphics[height=0.75\linewidth]{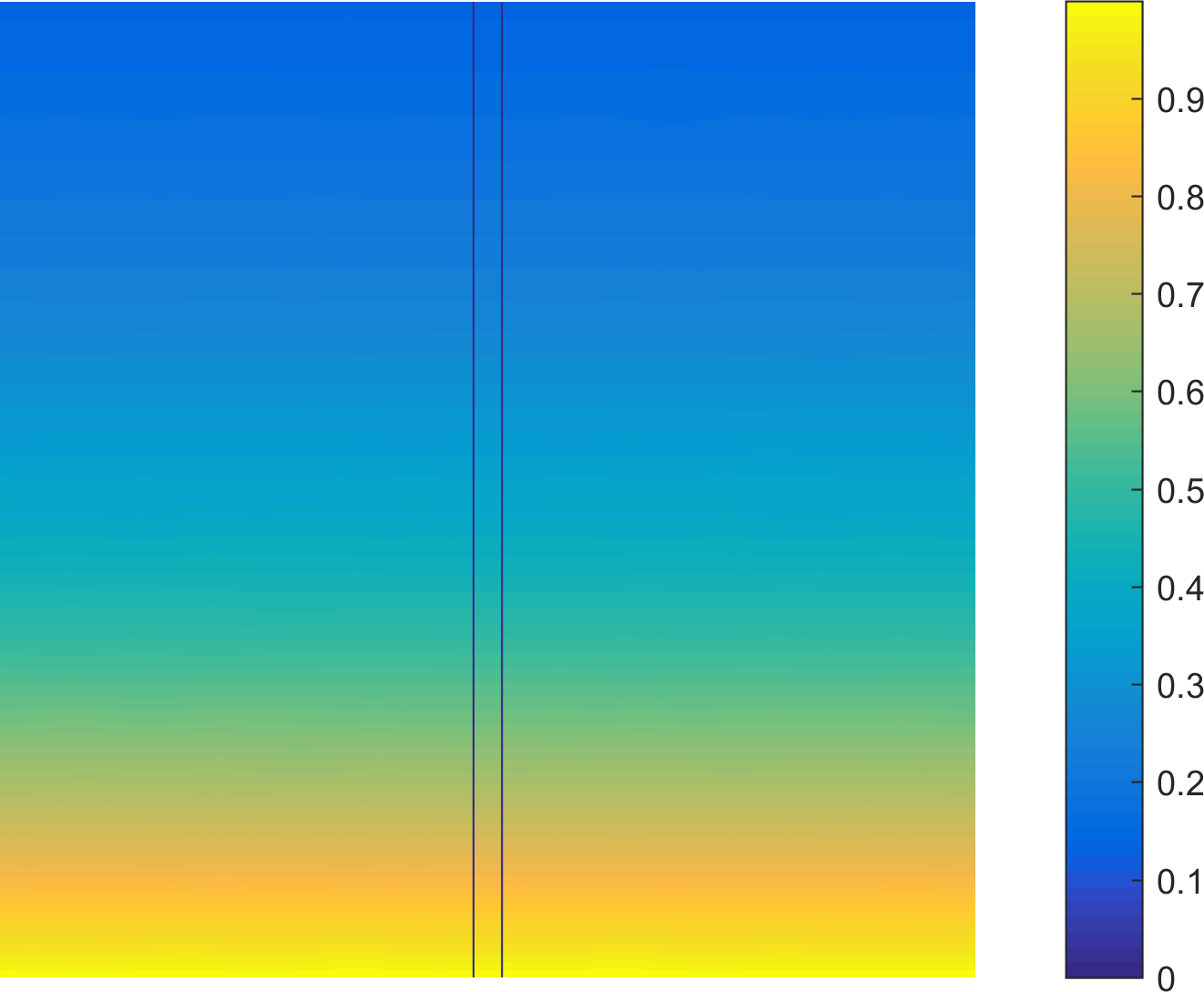}
\caption{Numerical solution}\label{fig:ex2-num}
\end{subfigure}
\caption{Crack with out-flow (Example 2).
(a) The set-up for this example is $\beta_{2,1}=[-1,0]$, $\beta_{2,2}=-\beta_{2,1}$ and in the crack $\beta_{1,1} = [0,1]$.
(b) The numerical solution corresponds well to the exact solution which is $u=e^{-2y}$ in all domains.
}
\label{fig:ex2}
\end{figure}

\begin{figure}
\centering
\begin{subfigure}[t]{.24\linewidth}\centering
\includegraphics[width=0.87\linewidth]{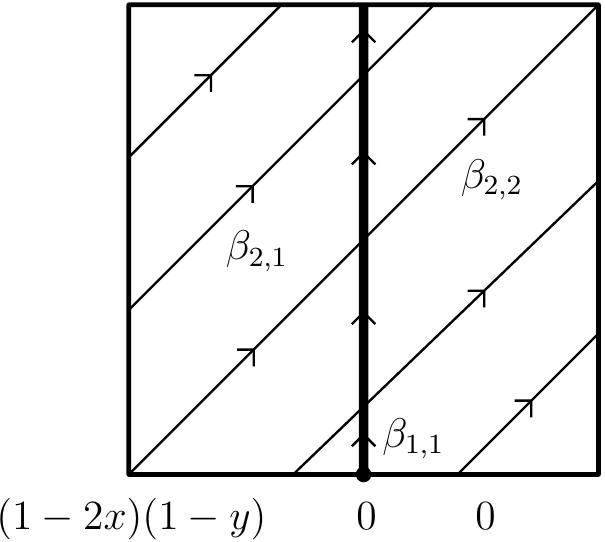}
\caption{Set-up} \label{fig:ex3-illustration}
\end{subfigure}
\begin{subfigure}[t]{.24\linewidth}\centering
\includegraphics[width=0.95\linewidth]{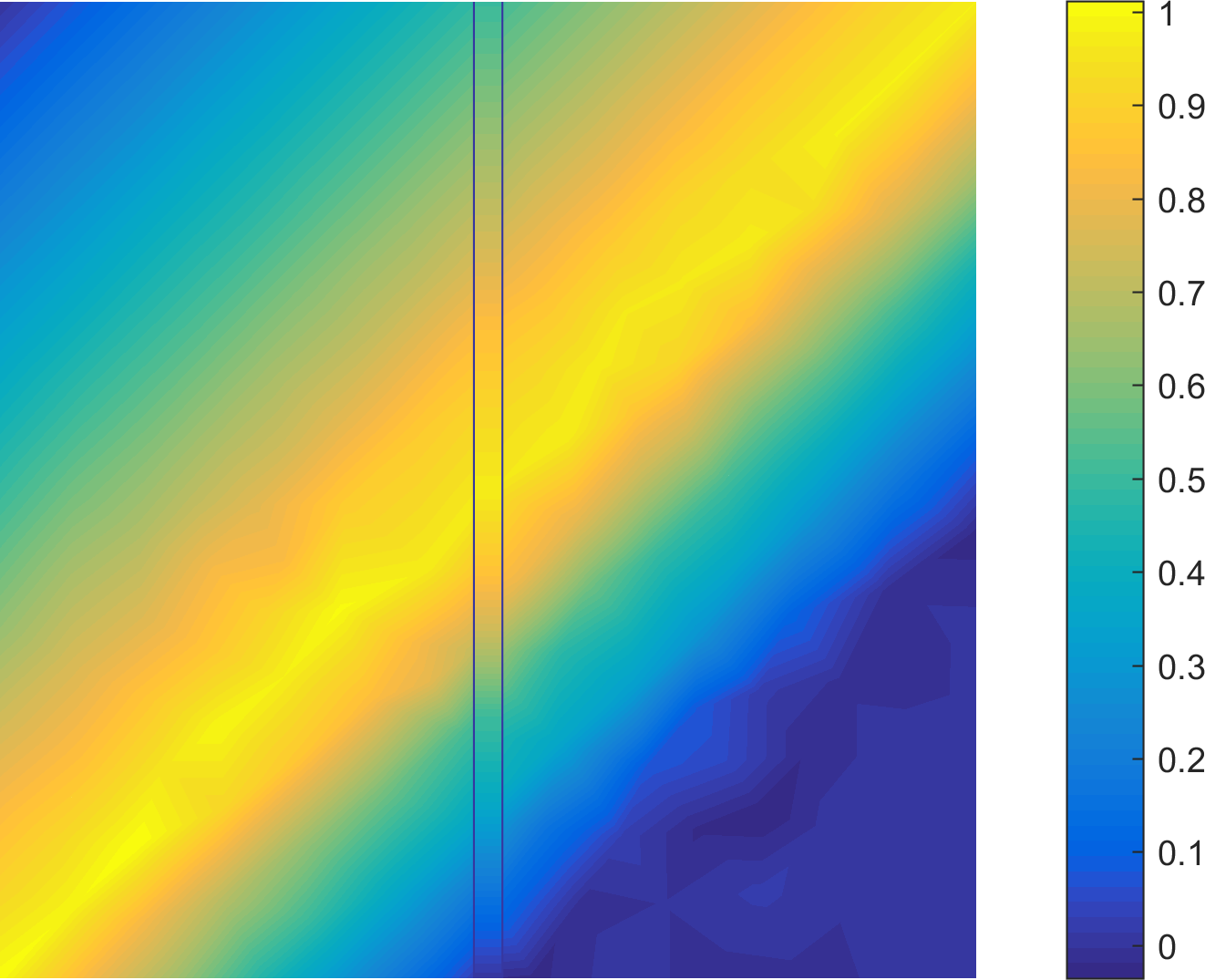}
\caption{$\beta_{1,1} = [0,0]$}\label{fig:ex3-num1}
\end{subfigure}
\begin{subfigure}[t]{.24\linewidth}\centering
\includegraphics[width=0.95\linewidth]{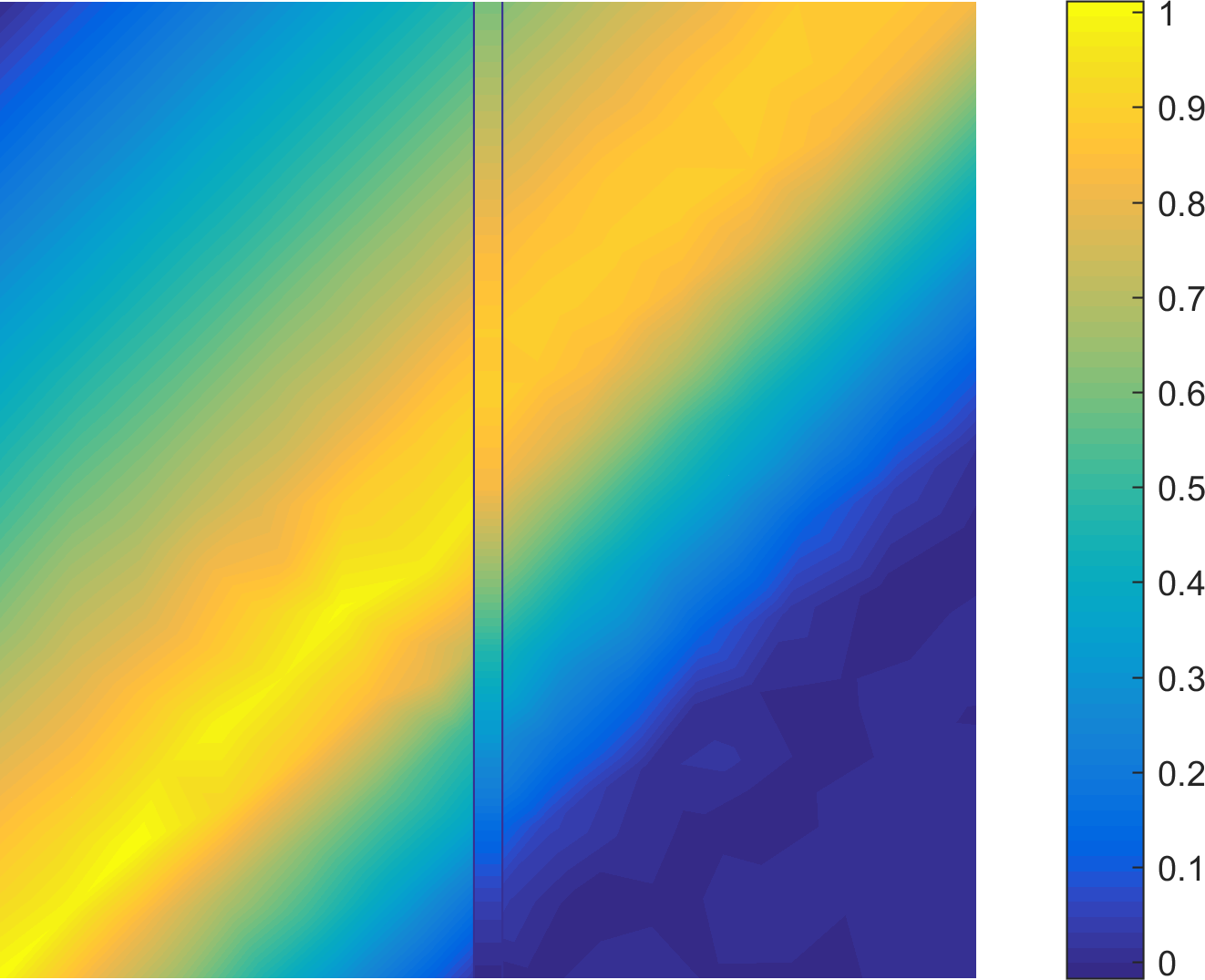}
\caption{$\beta_{1,1} = [0,0.1]$} \label{fig:ex3-num2}
\end{subfigure}
\begin{subfigure}[t]{.24\linewidth}\centering
\includegraphics[width=0.95\linewidth]{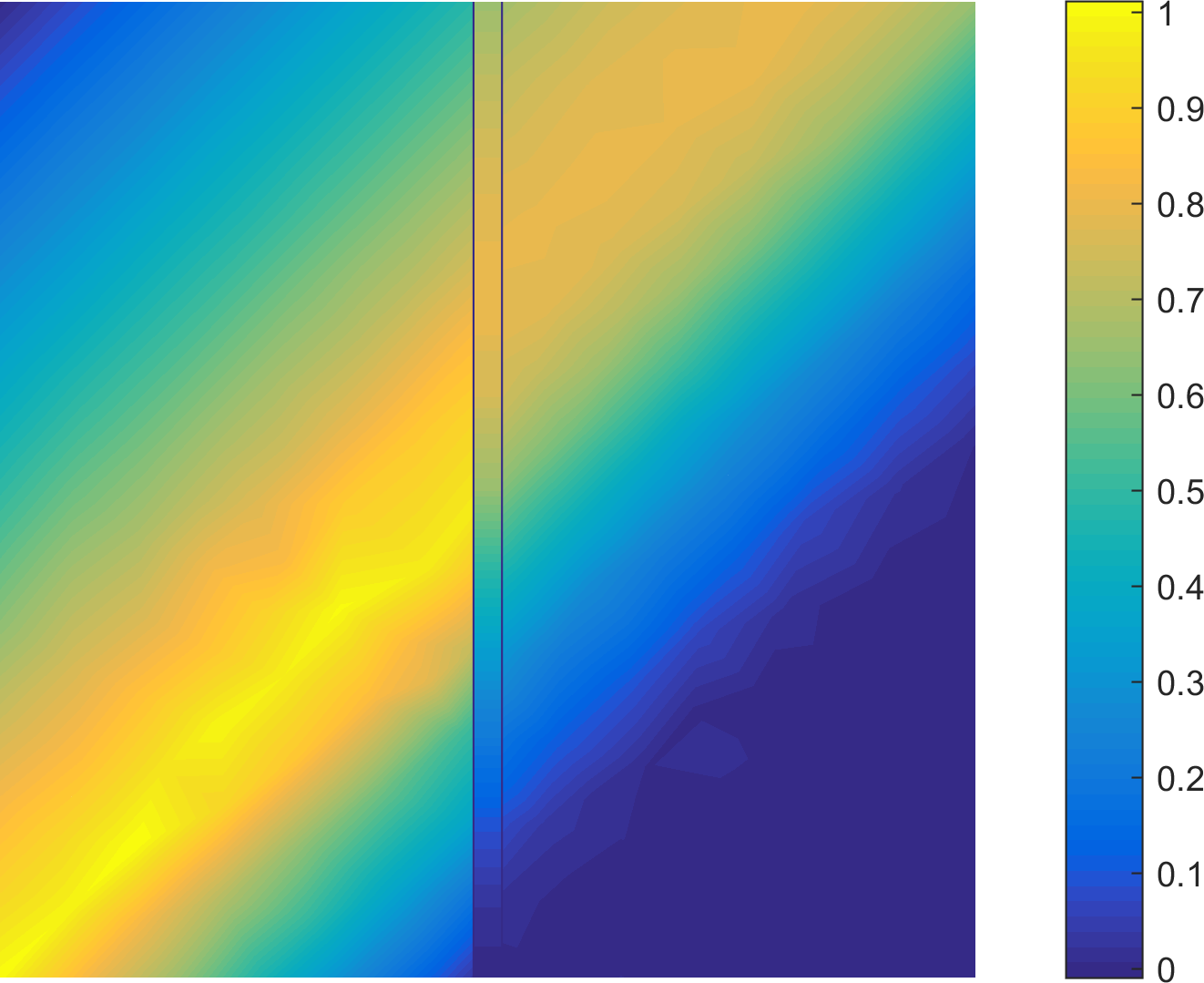}
\caption{$\beta_{1,1} = [0,0.2]$}\label{fig:ex3-num3}
\end{subfigure}
\caption{Flow crossing a crack (Example 3).
(a) The set-up for this example is $\beta_{2,1}=\beta_{2,2}=[1,1]$ and with $\beta_{1,1}$ in the crack varying in the numerical solutions (b)--(d) as specified by their captions. Note that in case (b) the solution is not affected by the presence of a crack.
}
\label{fig:ex3}
\end{figure}

\begin{figure}
\centering
\begin{subfigure}[t]{.25\linewidth}\centering
\includegraphics[width=0.9\linewidth]{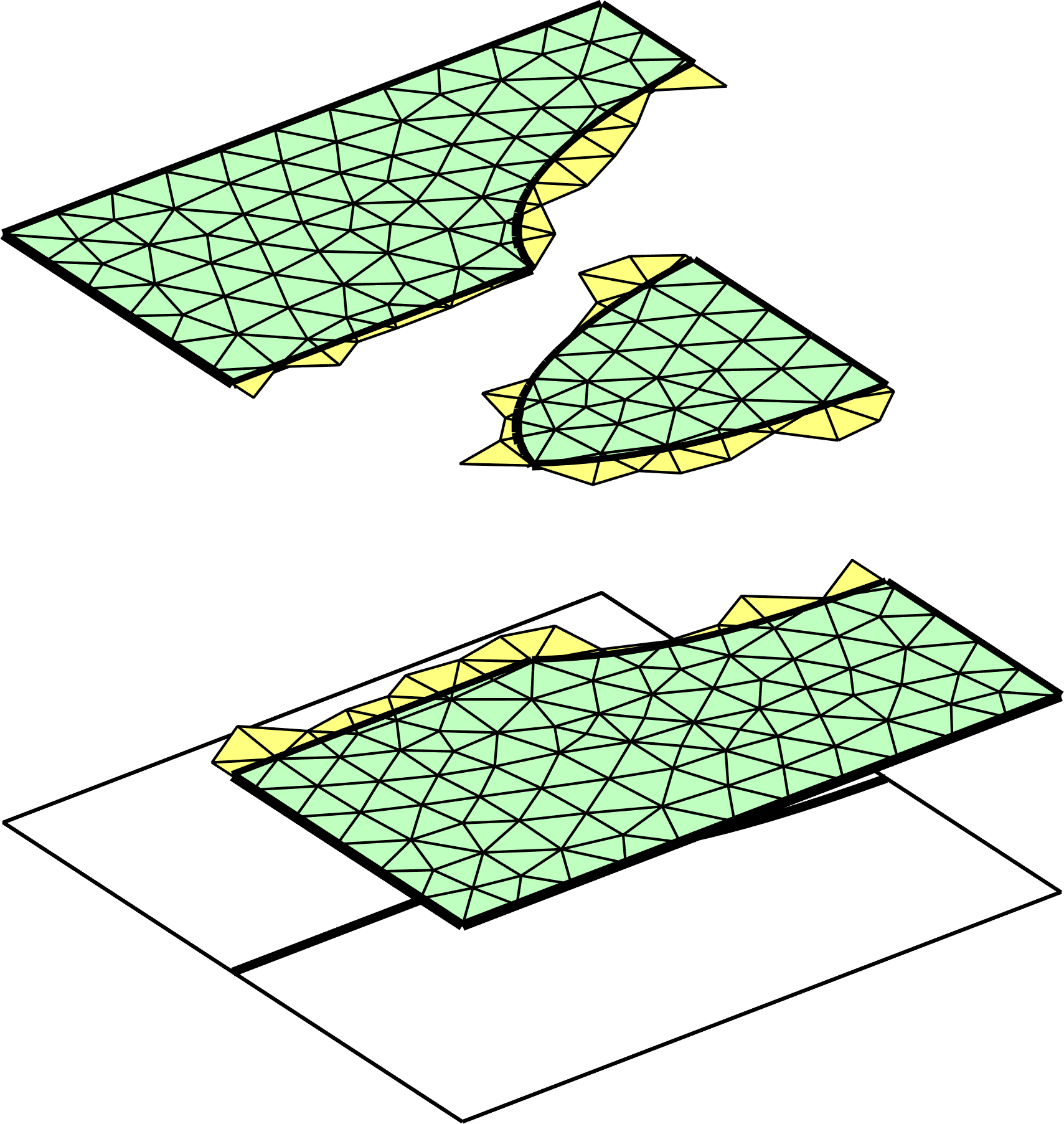}
\caption{$d=2$} \label{fig:case4-mesh-bulk}
\end{subfigure}
\begin{subfigure}[t]{.25\linewidth}\centering
\includegraphics[width=0.9\linewidth]{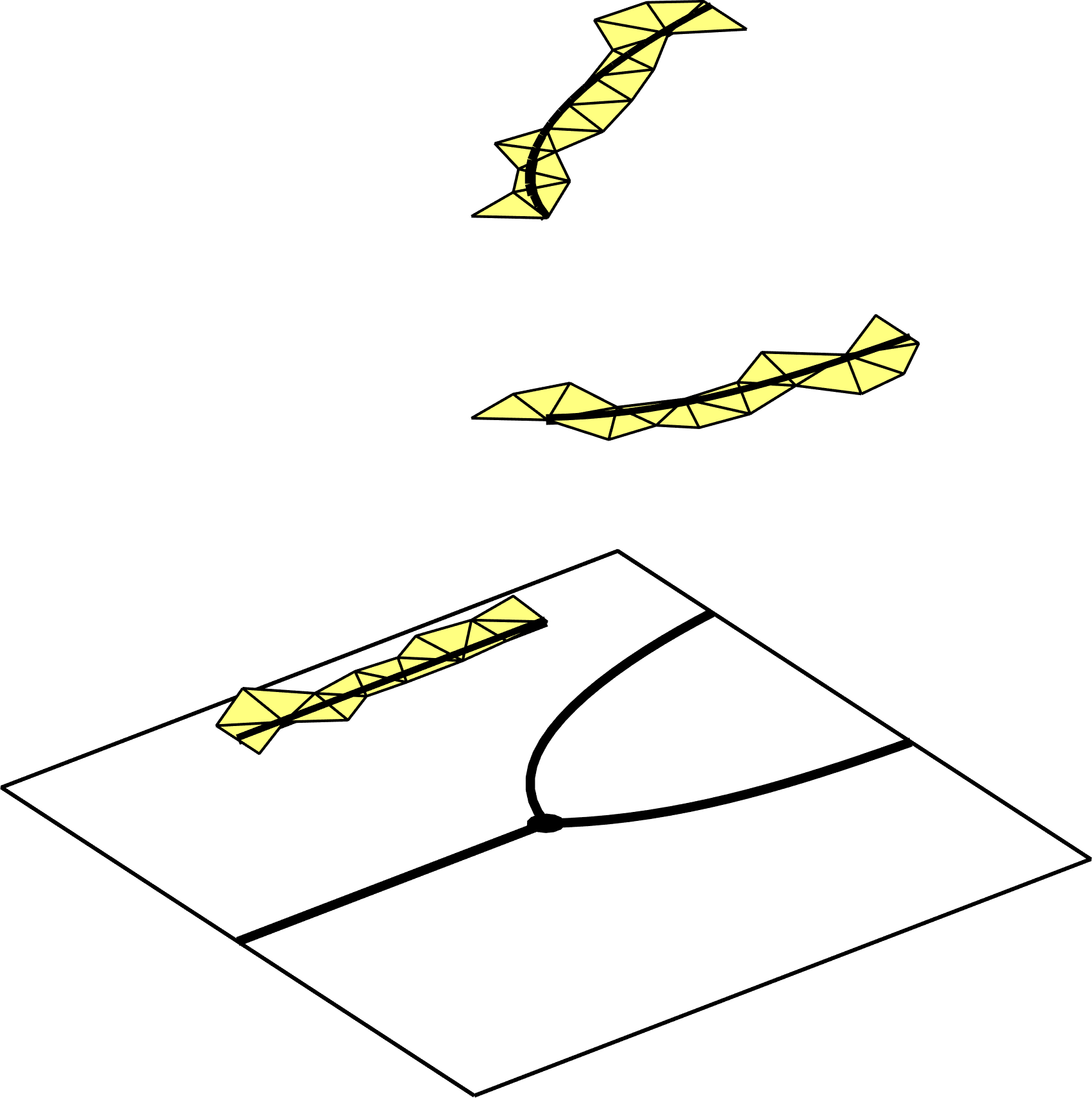}
\caption{$d=1$}\label{fig:case4-mesh-crack}
\end{subfigure}
\begin{subfigure}[t]{.25\linewidth}\centering
\includegraphics[width=0.9\linewidth]{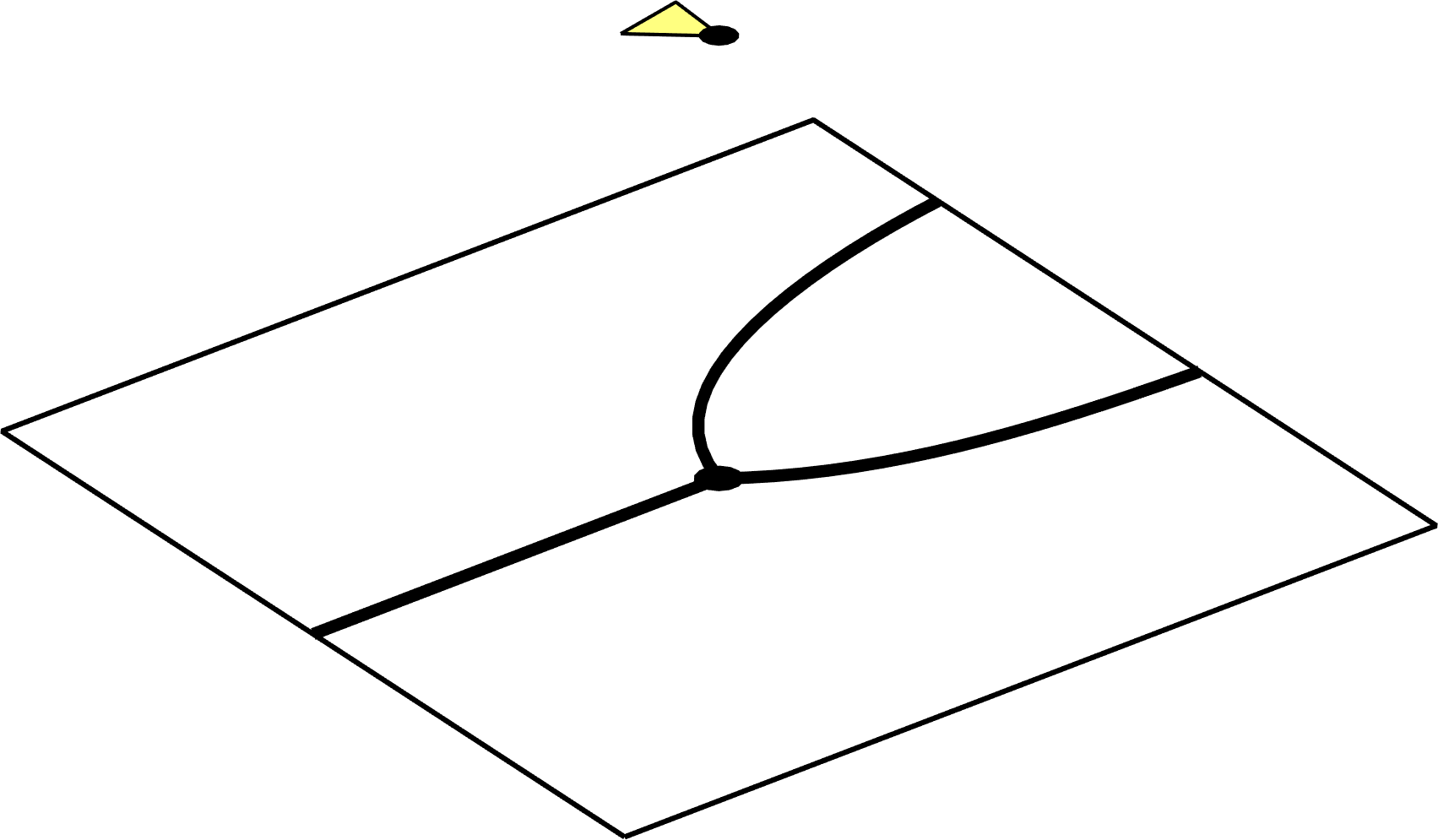}
\caption{$d=0$}\label{fig:case4-mesh-point}
\end{subfigure}
\caption{Active meshes used in Example 4 ($h=0.1$). (a) Meshes for the bulk domains. (b) Meshes for the cracks. (c) Mesh for the bifurcation point.}
\label{fig:ex4-mesh}
\end{figure}

\begin{figure}
\centering
\begin{subfigure}[t]{.3\linewidth}\centering
\includegraphics[height=0.75\linewidth]{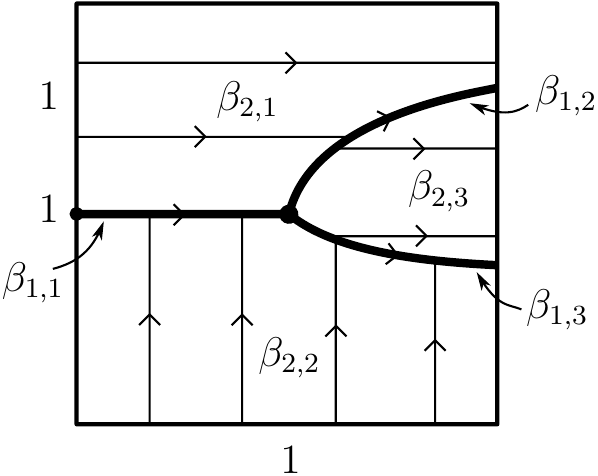}
\caption{Set-up} \label{fig:case4-illustration}
\end{subfigure}
\begin{subfigure}[t]{.3\linewidth}\centering
\includegraphics[height=0.75\linewidth]{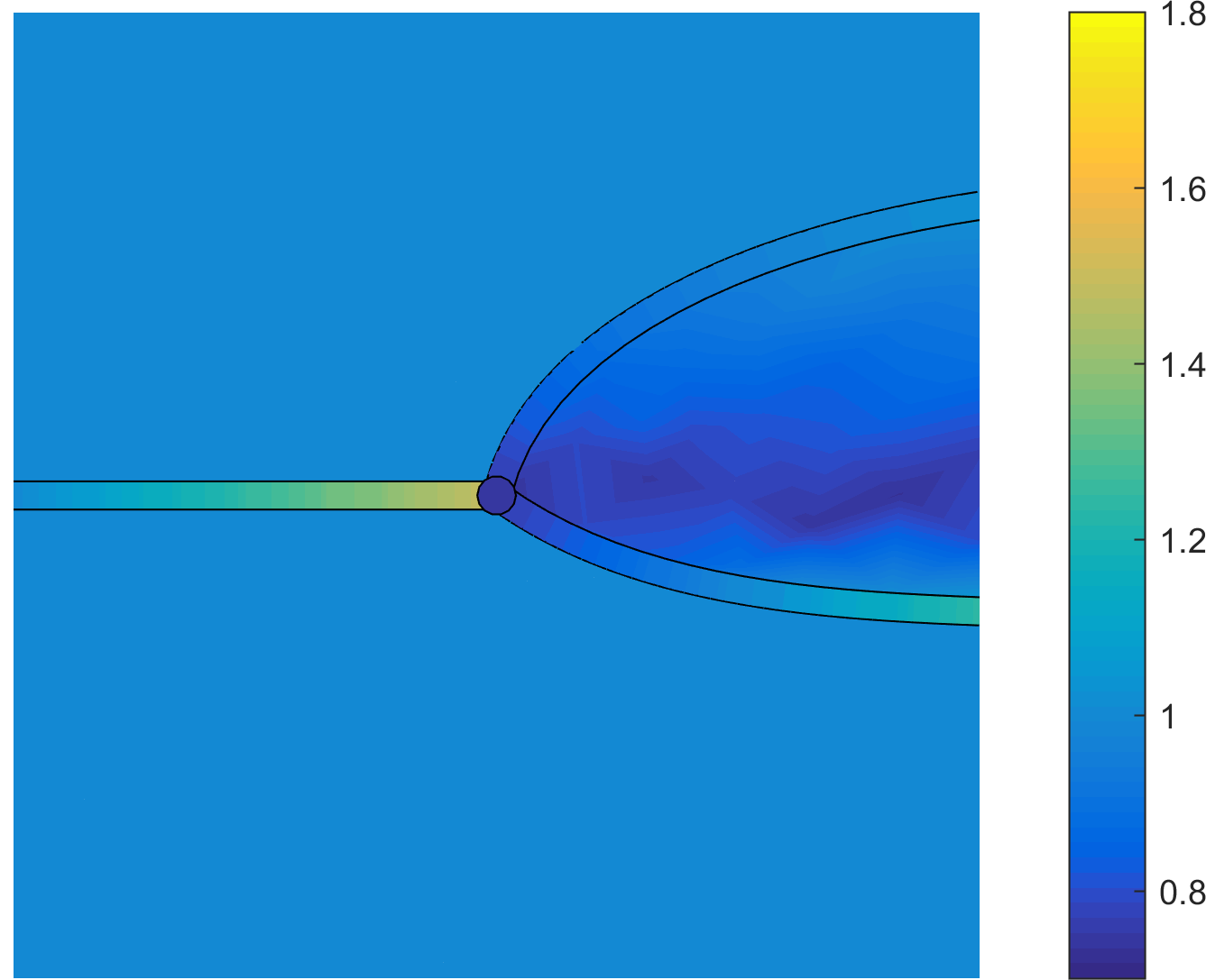}
\caption{$\beta_{1,2}=\beta_{1,3}=t$}\label{fig:case4-same}
\end{subfigure}
\begin{subfigure}[t]{.3\linewidth}\centering
\includegraphics[height=0.75\linewidth]{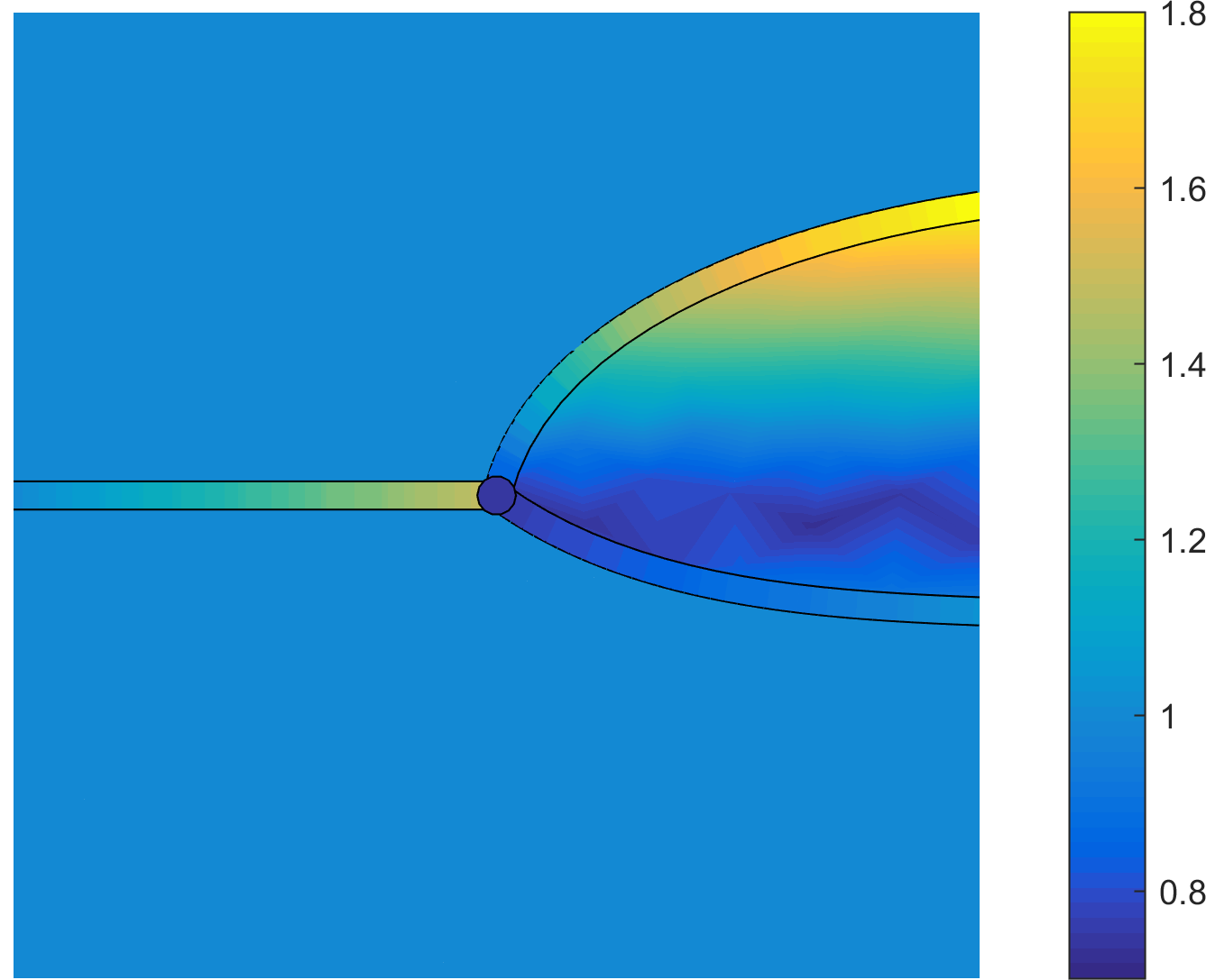}
\caption{$\beta_{1,2}=0.25t$, $\beta_{1,3}=1.75t$}\label{fig:case4-diff}
\end{subfigure}
\caption{Cracks with a bifurcation point (Example 4).
(a) Here $\beta_{2,1}=[1,0]$, $\beta_{2,2}=[0,1]$, $\beta_{2,3}=[0.1,0]$, $\beta_{1,1}=[1,0]$ while $\beta_{1,2}$ and $\beta_{1,3}$ changes between the examples. (b)--(c) Numerical solutions using values of $\beta_{1,2},\beta_{1,3}$ specified in the captions where $t$ is the unit tangent in the cracks.
}
\label{fig:ex4}
\end{figure}

\begin{figure}
\centering
\begin{subfigure}[t]{.3\linewidth}\centering
\includegraphics[height=0.75\linewidth]{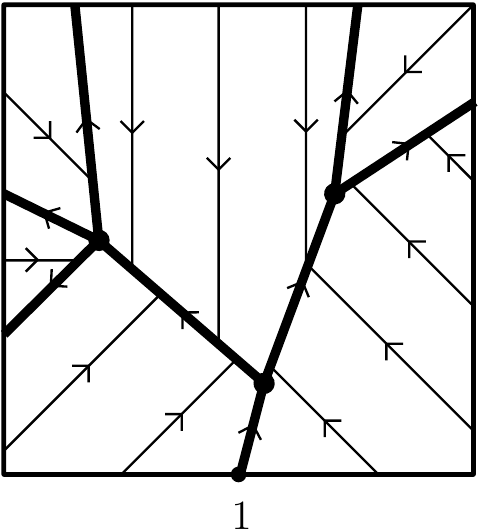}
\caption{Set-up} \label{fig:tree-illustration}
\end{subfigure}
\begin{subfigure}[t]{.3\linewidth}\centering
\includegraphics[height=0.75\linewidth]{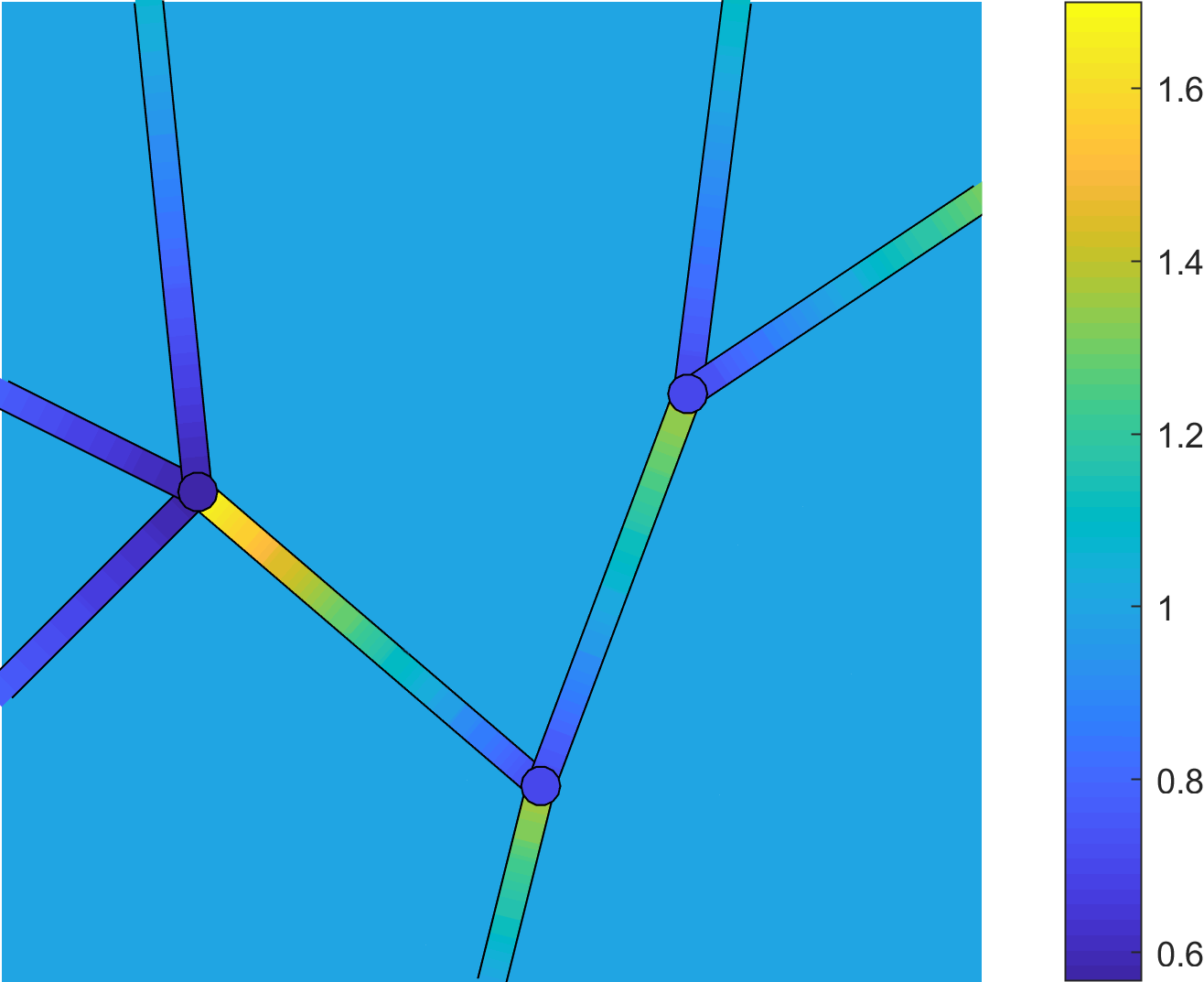}
\caption{Numerical solution}\label{fig:tree-num}
\end{subfigure}
\caption{System of cracks with in-flow (Example 5).
(a) Starting in the lower left corner and traversing the bulk domains clockwise $\beta_{2,i}$ is $[1,1]$, $[1,0]$, $[1,-1]$, $[0,-1]$, $[-1,-1]$ and $[-1,1]$. In the cracks $\beta_{1,i}$ is 1 in the tangent direction according to the figure. All boundary values are 1.
(b) In the numerical solution ($h=0.1$) we note that the crack solutions at each bifurcation point divide the in-flow solution equally among the out-flow solutions.
}
\label{fig:ex-tree}
\end{figure}
\newpage
\bibliographystyle{abbrv}
\footnotesize{
\bibliography{layer_biblio}
}

\bigskip
\bigskip
\noindent
\footnotesize {\bf Acknowledgements.}
This research was supported in part by the Swedish Foundation
for Strategic Research Grant No.\ AM13-0029, the Swedish Research
Council Grants Nos.\  2013-4708, 2017-03911, and the Swedish
Research Programme Essence. EB was supported by EPSRC research grants EP/P01576X/1 and EP/P012434/1.

\bigskip
\bigskip
\noindent
\footnotesize {\bf Authors' addresses:}

\smallskip
\noindent
Erik Burman,  \quad \hfill \addressuclshort\\
{\tt e.burman@ucl.ac.uk}

\smallskip
\noindent
Peter Hansbo,  \quad \hfill \addressjushort\\
{\tt peter.hansbo@ju.se}

\smallskip
\noindent
Mats G. Larson,  \quad \hfill \addressumushort\\
{\tt mats.larson@umu.se}

\smallskip
\noindent
Karl Larsson, \quad \hfill \addressumushort\\
{\tt karl.larsson@umu.se}

\end{document}